\documentclass[aop,MSNbibl,dvips]{arximspdf}

\doi{10.1214/12-AOP779} 
\volume{41}
\issue{5}
\pubyear{2013}
\firstpage{3392}
\lastpage{3419}

\makeatletter
\newcommand{\rrVert}{\Vert}
\newcommand{\rrvert}{\vert}
\newcommand{\llVert}{\Vert}
\newcommand{\llvert}{\vert}
\newcommand{\eqref}[1]{(\ref{#1})}
\newtheorem{theorem}{Theorem}[section]

\newtheorem{conjecture}[theorem]{Conjecture}
\newtheorem{question}[theorem]{Question}
\newtheorem{lemma}[theorem]{Lemma}
\newtheorem{cor}[theorem]{Corollary}
\newproclaim{defn}[theorem]{Definition}
\newproclaim{fact}[theorem]{Fact}

\newproclaim{remark}[theorem]{Remark}

\def\pr{\mathbb{P}}

\renewcommand{\epsilon}{\varepsilon}
\newcommand{\eps}{\varepsilon}
\newcommand{\1}{\mathbf{1}}
\newcommand{\Lip}{\mathrm{Lip}}
\newcommand{\orbit}{\setminus}

\newcommand{\Aut}{\operatorname{\mathsf{Aut}}}
\newcommand{\dist}{\operatorname{\mathsf{dist}}}
\renewcommand{\dim}{\operatorname{\mathsf{dim}}}

\newcommand{\E}{\mathbb{E}}
\newcommand{\remove}[1]{}
\renewcommand{\le}{\leq}
\renewcommand{\ge}{\geq}

\makeatother

\begin{document}
\begin{frontmatter}

\title{Harmonic maps on amenable groups and a~diffusive lower bound
for random walks}
\runtitle{Harmonic maps and random walks}

\begin{aug}
\author[A]{\fnms{James R.} \snm{Lee}\corref{}\thanksref{t1}\ead[label=e1]{jrl@cs.washington.edu}}
\and
\author[B]{\fnms{Yuval} \snm{Peres}\ead[label=e2]{peres@microsoft.com}}
\thankstext{t1}{Supported by NSF Grant CCF-0644037 and a
Sloan Research Fellowship.}
\runauthor{J. R. Lee and Y. Peres}
\affiliation{University of Washington and Microsoft Research}
\address[A]{Department of Computer Science and Engineering\\
Box 352350\\
University of Washington\\
Seattle, Washington 98195-2350\\
USA\\
\printead{e1}} 
\address[B]{Microsoft Research\\
One Microsoft Way\\
Washington, 98052\\
USA\\
\printead{e2}}
\end{aug}

\received{\smonth{9} \syear{2011}}
\revised{\smonth{5} \syear{2012}}

%
\begin{abstract}
We prove diffusive lower bounds on the rate of escape of the random walk
on infinite transitive graphs.
Similar estimates hold for finite graphs, up to the relaxation time
of the walk. Our approach uses nonconstant equivariant
harmonic mappings taking values in a Hilbert space.
For the special case of discrete, amenable groups,
we present a more explicit proof
of the Mok--Korevaar--Schoen theorem on the existence of such harmonic
maps by
constructing them from the heat flow on a F{\o}lner set.
\end{abstract}

%

\end{frontmatter}

\section{Introduction}
\label{secintro}

Let $G$ be a $d$-regular, transitive graph (i.e., with transitive
automorphism group),
let $\{X_t\}$ denote the symmetric simple random walk on $G$ with $X_0$
arbitrary
and let $\dist$ be the path metric on $G$.
In the case when $G$ is a Cayley graph of a finitely-generated,
amenable group, \`{E}rshler~\cite{Erschler08} showed that
$\mathbb E [\dist(X_0,X_t)^2 ] \geq C t/d$ for all times $t
\geq1$, where
$C > 0$ is some absolute constant.

Our first theorem concerns a more precise analysis of the random walk behavior,
as well as an extension to general transitive, amenable graphs. Recall that
a graph~$G$ is \textit{amenable} if there exists a sequence
of finite subsets $\{S_j\}$ of the vertices such that $|S_j \triangle
N(S_j)|/|S_j| \to0$,
where $N(S_j)$ denotes the neighborhood of $S_j$ in~$G$.

\begin{theorem}\label{thmmaininfinite}
Suppose $G$ is an infinite, connected, and amenable transitive
$d$-regular graph. Then the simple random walk on $G$ satisfies the estimate
\[
\E \bigl[\dist(X_0,X_t)^2 \bigr] \geq t/d .
\]
Moreover, for some universal constants $C > 0$ and $C' \geq1$, and $t
\geq d$, we have the estimates
\[
\E \bigl[\dist(X_0, X_t) \bigr] \geq C \sqrt{t/d},
\]
and for every $\varepsilon\geq1/\sqrt{t}$,
\[
\frac{1}{t} \sum_{s=0}^t \pr \bigl[
\dist(X_0, X_s) \leq\varepsilon \sqrt{t/d} \bigr] \leq
C' \varepsilon .
\]
\end{theorem}

In Section~\ref{secapplications},
we prove a version of the preceding theorem
for the Cayley graph of any group without property (T).

In various senses, Theorem~\ref{thmmaininfinite} shows that among
infinite transitive graphs, the random walk
disperses slowest for the standard random walk on $\mathbb Z$; see
Corollary~\ref{coramen} and Remark~\ref{remasymptotic}.
We also prove a version for finite graphs which holds up to the relaxation
time of the random walk. In this case, the bound is matched (up to
constant factors) for the finite cycle graphs; see Remark~\ref{remweighted}.

\begin{theorem}\label{thmmainfinite}
Suppose $G$ is a finite, connected, transitive $d$-regular graph, and
$\lambda$ denotes the second-largest
eigenvalue of the transition matrix $P$ of the random walk on $G$.
Then for every $t \leq(1-\lambda)^{-1}$,
\[
\E \bigl[\dist(X_0,X_t)^2 \bigr] \geq
t/(2d) .
\]
Moreover, for some universal constants $C > 0$ and $C' \geq1$, and all
$t$ such that $(1-\lambda)^{-1} \geq t \geq d$, we have the estimates
\[
\E \bigl[\dist(X_0, X_t) \bigr] \geq C \sqrt{t/d},
\]
and for every $\varepsilon\geq1/\sqrt{t}$,
\[
\frac{1}{t} \sum_{s=0}^t \pr \bigl[
\dist(X_0, X_s) \leq\varepsilon \sqrt{t/d} \bigr] \leq
C' \varepsilon .
\]
\end{theorem}

We remark that, in both cases, the dependence on $d$ is necessary; see
Remark~\ref{remweighted}.

%

The proof of Theorem~\ref{thmmaininfinite} is based on the existence
of nonconstant,
equivariant harmonic maps on transitive, amenable graphs.
For the simplicity of presentation, we first restrict ourselves to the
setting of groups.
Let $\Gamma$ be a group with finite generating set $S \subseteq\Gamma$,
and let $G$ be the corresponding Cayley graph. Here and throughout the paper,
all Cayley graphs will be defined using multiplication by the
generators on the right.

Suppose that $\mathcal H$ is some Hilbert space
on which $\Gamma$ acts by isometries, and we have a nonconstant equivariant
harmonic map $\Psi: \Gamma\to\mathcal H$, that is, such that $g \Psi
(h) = \Psi(gh)$
and $\Psi(h) = |S|^{-1} \sum_{s \in S} \Psi(hs)$ hold for every $h
\in\Gamma$.
\`{E}rshler~\cite{Erschler08} observed that this can be used
to lower bound $\E [\dist(X_0,X_t)^2 ]$, as follows.

We may normalize $\Psi$ so that, if $e \in\Gamma$ is the identity,
%
\begin{equation}
\label{eqnormalize} \frac{1}{|S|} \sum_{s \in S} \bigl\|
\Psi(e)-\Psi(s)\bigr\| ^2=1.
\end{equation}

By equivariance, this implies that $\Psi$ is $\sqrt{|S|}$-Lipschitz, hence
\[
\E \bigl[\dist(X_0,X_t)^2 \bigr] \geq
\frac{1}{|S|} \E \bigl\| \Psi (X_0)-\Psi(X_t)
\bigr\| ^2.
\]
But since $\Psi$ is harmonic, $\Psi(X_t)$ is a martingale, thus
\[
\E \bigl\| \Psi(X_0)-\Psi(X_t)\bigr\| ^2 = \sum
_{j=0}^{t-1} \E\bigl \| \Psi (X_j)-
\Psi(X_{j+1})\bigr\| ^2 = t,
\]
where in the final line we have used equivariance and \eqref{eqnormalize}.

By results of Mok~\cite{Mok95} and Korevaar and Schoen~\cite{KS97},
if $\Gamma$ is amenable,
then it always admits such an equivariant harmonic map. On the other hand,
if $\Gamma$ is not amenable, then $G$ has spectral radius $\rho< 1$
\cite{Kesten59},
hence $\E[\dist(X_0,X_t)^2] \geq C t^2$,
for some constant $C=C(\rho) > 0$; see, for example,~\cite{Woess00}, Proposition 8.2.
Thus the preceding discussion shows that $\E[\dist(X_0,X_t)^2]$ grows
at least linearly
in $t$, for any infinite group $\Gamma$.

In Section~\ref{secescape}, we exhibit a general method for proving
escape lower bounds. For any function $\psi\in\ell^2(\Gamma)$, we have
%
\begin{equation}
\E \bigl[\dist(X_0,X_t)^2 \bigr] \geq
\frac{1}{d} \biggl(t - t^2 \frac{\| (I-P) \psi\| ^2}{2\langle\psi, (I-P) \psi\rangle} \biggr),
\end{equation}
where $P$ is the transition kernel of the random walk on $G$.
For finite groups, we choose $\psi$ to be the eigenfunction corresponding
to the second-largest eigenvalue of $P$.
For infinite amenable groups, one can obtain $\psi$ directly
from spectral projection.

For a more explicit approach
in the infinite, amenable case, we show that one can obtain the $\E
[\dist(X_0,X_t)^2] \geq t/|S|$ bound
by taking a sequence of functions $\{\psi_n\}$ to be a truncated heat
flow from some sets $A_n \subseteq\Gamma$,
that is, $\psi_n = \sum_{i=0}^{n} P^i \1_{A_n}$, where $\{A_n\}$ forms an
appropriate F{\o}lner sequence in $G$.
These lower bounds, and indeed all the results in our paper,
are proved for amenable, transitive graphs (and even quasi-transitive graphs),
and more general forms of stochastic processes.

The existence of nonconstant, equivariant, harmonic maps on groups
without property (T) is proved in~\cite{Mok95,KS97}; see also~\cite{Kleiner07}, Appendix A, for an exposition in the case of discrete
groups, based on~\cite{FM05}.
In Section~\ref{secharmonic}, inspired by the preceding escape lower
bounds, we give an explicit construction
of these harmonic maps, simple enough to describe here. We focus now
on the amenable case; in Theorem~\ref{thmharmonicNPT}, we show that
this approach generalizes to all discrete groups without property~(T).

Define $\Psi_n \dvtx \Gamma\to\ell^2(\Gamma)$ by
\[
\Psi_n(x) \dvtx g \mapsto\frac{\psi_n(gx)}{\sqrt{2\langle\psi_n,
(I-P)\psi_n\rangle}} ,
\]
where $\psi_n$ is as before.

We argue that, after applying an appropriate affine isometry to each
$\Psi_n$,
there is a subsequence of $\{\Psi_n\}$ which converges pointwise
to a nonconstant, equivariant, harmonic map.
Our construction works for any infinite, transitive, amenable graph;
see Theorem~\ref{thmharmonic}.

\begin{theorem}
Let $G=(V,E)$ be any infinite, connected, amenable, transitive graph. Then
there exists a Hilbert space $\mathcal H$
and an $\mathcal H$-valued,
nonconstant equivariant harmonic mapping on $G$, where equivariance
is understood with respect to the transitive action on $G$.
\end{theorem}

In Section~\ref{secwithoutT}, we show that our approach
also proves the preceding theorem for the
Cayley graph of any group without property (T).
It is known~\cite{Mok95,KS97} that a group
admits such an equivariant harmonic mapping
if and only if it does not have property (T); see also~\cite{Kleiner07}, Lemma A.6.

One can use such mappings to obtain more detailed information on the
random walk.
Vir\'ag~\cite{Virag08} showed that, in the setting of
Cayley graphs of amenable groups, one has
$\E [\dist(X_0,X_t) ] \geq C\sqrt{t}/|S|^{3/2}$
for some $C > 0$.
This can be proved by analyzing the process $\Psi(X_t)$ via the BDG
martingale inequalities; see, for example,~\cite{KW92}, Theorem 5.6.1.\setcounter{footnote}{1}\footnote{In fact, Virag proceeds
by explicitly bounding
$\E [\| M_0-M_t\| ^4 ] \leq O(|S|^2 t^2)$ when $\{M_t\}$ is
any Hilbert space-valued martingale with $\E [\| M_{t+1}-M_t\| ^2
|  \mathcal F_t ]\leq1$
and $\E [\| M_{t+1}-M_t\| ^4 |  \mathcal F_t ] \leq|S|^2$,
for all $t \geq0$.}

In Section~\ref{secestimates},
we show how a stronger bound can
be derived directly from hitting time estimates,
which can themselves be easily derived for martingales, then transferred
to the group setting via harmonic maps. More generally,
we study some finer properties of the escape behavior
of the random walk.

\subsection{Related work}

We recall some previous results on the rate of escape
of random walks on groups.
A large amount of work has been devoted to classifying
situations where the rate of escape $\mathbb E [\dist(X_0,X_t)]$ is linear;
we refer to the survey of Vershik~\cite{Vershik00}.
\`{E}rshler has given examples where the rate
can be asymptotic to\vadjust{\goodbreak}
$t^{1-2^{-k}}$ for any $k \in\mathbb N$
\cite{Ershler01}. Following seminal work of Varopoulos~\cite{Var},
Hebisch and Saloff-Coste~\cite{HSC} obtained precise heat kernel
estimates for symmetric bounded-range random walks on groups of
polynomial growth. In particular, Theorem 5.1 in~\cite{HSC} implies
our Theorem~\ref{thmmaininfinite} for groups of polynomial growth.
However, for groups of super-polynomial growth, it seems that existing
heat-kernel bounds (see, e.g., Theorem 4.1 in~\cite{HSC}) are not
powerful enough to imply Theorem~\ref{thmmaininfinite}.
Diaconis and Saloff-Coste~\cite{DSC94} show that
on finite groups satisfying
a certain ``moderate growth'' condition,
the random walk mixes in $O(D^2)$ steps,
where $D$ is the diameter of the group in the word metric.
A sequence of works~\cite{ANP,NP1,NP2}
have related the rate of escape of random walks
to questions in geometric group theory,
notably to estimates of Hilbert compression
exponents of groups. Our argument for finite groups was motivated by
the work of the first author with Y.~Makarychev~\cite{LM} on effective,
finitary versions of Gromov's polynomial growth theorem. Another
substantial work in this direction is the recent preprint of Shalom and
Tao~\cite{ST},
written independently of the present paper. Constructions of nearly
harmonic functions play a key role there as well.

\section{Escape rate of random walks}
\label{secescape}


In the present section, we will consider a finite or infinite
symmetric, stochastic matrix $\{P(x,y)\}_{x,y \in V}$ for some
index set~$V$. 
We write $\Aut(P)$ for the set of all bijections $\sigma: V \to V$
whose diagonal action preserves $P$, that is, $P(x,y)=P(\sigma x,
\sigma y)$
for all $x,y \in V$. For the most part, we will be concerned
with matrices $P$ for which $\Aut(P)$ acts transitively on~$V$.
A~primary example is given by taking $P$ to be the transition matrix
of the simple random walk on a finite or infinite vertex-transitive
graph $G$.

\begin{theorem}\label{thmpotential}
Let $V$ be an at most countable index set, and
consider any symmetric, stochastic matrix
$\{P(x,y)\}_{x,y \in V}$.
Suppose that $\Gamma\leq\Aut(P)$ is a closed,
unimodular subgroup
which acts transitively on $V$, and let $G=(V,E)$
be any graph on which $\Gamma$ acts by automorphisms.
If $\dist$ is the path metric on $G$, and $\psi\in\ell^2(V)$,
then
%
\begin{equation}
\label{eqinfdrift} \E \bigl[\dist(X_0, X_t)^2
\bigr] \geq p_* \frac{\langle\psi,
(I-P^t) \psi\rangle}{\langle\psi, (I-P) \psi\rangle} \geq p_* \biggl(t - t^2
\frac{\| (I-P) \psi\| ^2}{2\langle\psi, (I-P) \psi
\rangle} \biggr),
\end{equation}
where $\{X_t\}$ denotes the random walk with transition kernel $P$
started at any $X_0=x_0 \in V$, and
\[
p_* = \min \bigl\{ P(x,y) \dvtx \{x,y\} \in E \bigr\}.
\]
\end{theorem}

\begin{pf}
Since $\Gamma$ is unimodular, we can choose the Haar measure $\mu$
on $\Gamma$ to be normalized
so that $\mu(\Gamma_x)=1$ for every $x \in V$,
where $\Gamma_x$ is the stabilizer subgroup of $x$.
(Note that the stabilizer $\Gamma_x$ is compact
since $\Gamma$ acts by automorphisms on $G$, which
has all its vertex degrees bounded by $1/p_*$.)
Define $\Psi: V \to L^2(\Gamma, \mu)$ by
$\Psi(x) \dvtx \sigma\mapsto\psi(\sigma x)$.\vadjust{\goodbreak}

In this case, for every $z \in V$,
%
\begin{eqnarray}
\label{eqlastbound}
&&\sum_{y \in V} P(z,y) \bigl\| \Psi(y)-
\Psi(z) \bigr\| _{L^2(\Gamma,\mu)}^2 \nonumber\\
&&\qquad= \sum_{y \in V}
P(z,y) \int \bigl|\psi(\sigma z)-\psi(\sigma y)\bigr|^2 \,d\mu(\sigma)
\nonumber
\\
&&\qquad= \sum_{y \in V} \int P(\sigma z, \sigma y) \bigl|\psi(
\sigma z)-\psi (\sigma y)\bigr|^2 \,d\mu(\sigma)
\\
&& \qquad=\mu(\Gamma_z) \sum_{x,y \in V} P(x,y) \bigl|
\psi(x)-\psi(y)\bigr|^2\nonumber
\\
&&\qquad= 2 \bigl\langle\psi, (I-P)\psi \bigr\rangle.\nonumber
\end{eqnarray}

Thus for $\{x,y\} \in E$, we have $\| \Psi(x)-\Psi(y)\| _{L^2(\Gamma
,\mu)}^2 \leq\frac{2 \langle\psi, (I-P)\psi\rangle}{p_*}$,
which implies that
%
\begin{equation}
\label{eqgrad} \| \Psi\| _{\Lip} \leq\sqrt{\frac{2 \langle\psi, (I-P)\psi\rangle}{p_*}},
\end{equation}
where $\Psi$ is considered as a map from $(V,\dist)$ to $L^2(\Gamma
,\mu)$,
and we use $\| \Psi\| _{\Lip}$ to denote the infimal number $L$ such
that $\Psi$
is $L$-Lipschitz.

So, for
any $x_0 \in V$, we have
%
\begin{eqnarray}\label{eqvariation}
&&\| \Psi\| _{\Lip}^2 \mathbb E \bigl[\dist(X_0,X_t)^2
| X_0=x_0 \bigr] \nonumber\\
&&\qquad\geq \mathbb E \bigl[\bigl\|
\Psi(X_0)-\Psi(X_t)\bigr\| _{L^2(\Gamma,\mu)}^2 |
X_0=x_0 \bigr]
\nonumber
\\
&&\qquad= \int\mathbb E \bigl[\bigl|\psi(\sigma X_0)-\psi(\sigma
X_t)\bigr|^2 | X_0=x_0 \bigr] \,d
\mu( \sigma)
\nonumber
\\[-8pt]
\\[-8pt]
\nonumber
&&\qquad= \sum_{x \in V} \mathbb E \bigl[\bigl|
\psi(X_0)-\psi(X_t)\bigr|^2 | X_0=x
\bigr]
\\
&&\qquad= 2 \bigl\langle\psi, \bigl(I-P^t \bigr) \psi \bigr\rangle
\nonumber
\\
&&\qquad= 2 \sum_{i=0}^{t-1} \bigl\langle\psi,
(I-P) P^i \psi \bigr\rangle, \nonumber
\end{eqnarray}
where in the third line, we have used the fact that the action of
$\sigma$
preserves~$P$.

To finish, we use the fact that $I-P$ is self-adjoint to compare
adjacent terms via
\begin{eqnarray*}
&&\bigl| \bigl\langle\psi, (I-P) P^{i} \psi \bigr\rangle- \bigl\langle\psi,
(I-P) P^{i+1} \psi \bigr\rangle\bigr|\\
&&\qquad = \bigl| \bigl\langle(I-P) \psi,
P^i (I-P) \psi \bigr\rangle\bigr| \leq\bigl\| (I-P) \psi\bigr\| ^2,
\end{eqnarray*}
where the final inequality follows because $P$ is stochastic, and
hence a contraction. From this, we infer that $\langle\psi, (I-P)P^i
\psi\rangle\ge\langle\psi, (I-P) \psi\rangle-i\| (I-P) \psi\| ^2$, whence
\[
\sum_{i=0}^{t-1} \bigl\langle\psi,
(I-P)P^i \psi \bigr\rangle\geq t \bigl\langle \psi, (I-P) \psi \bigr
\rangle- \frac{t^2}{2} \bigl\| (I-P)\psi\bigr\| ^2.
\]
Combining the preceding line with \eqref{eqgrad} and \eqref
{eqvariation} yields
\[
\frac{1}{p_*} \E \bigl[\dist(X_0,X_t)^2
\bigr] \geq\frac{\langle
\psi, (I-P^t) \psi\rangle}{\langle\psi, (I-P)\psi\rangle} \geq t - t^2 \frac{\| (I-P) \psi\| ^2}{2\langle\psi, (I-P) \psi\rangle}.
\]
\upqed\end{pf}

We now demonstrate circumstances
in which an appropriate $\psi\in\ell^2(V)$ exists.
Corollaries~\ref{corfinite},~\ref{corinfinite-drift}
and Conjecture~\ref{conjfinite}
all assume the notation of Theorem~\ref{thmpotential}.

\begin{cor}[(The finite case)]
\label{corfinite}
Let $V$ be a finite index set,
and
suppose that $\Aut(P)$ acts transitively on $V$.
If
$\lambda< 1$ is the second-largest eigenvalue of $P$, then
%
\begin{equation}
\label{eqsumbound} \mathbb E \bigl[\dist(X_0, X_t)^2
\bigr] \geq p_* \bigl(1+\lambda+\lambda^2+\cdots+
\lambda^{t-1} \bigr).
\end{equation}
In particular,
\[
\mathbb E \bigl[\dist(X_0, X_t)^2 \bigr]
\geq p_* t/2
\]
for $t \leq(1-\lambda)^{-1}$.
\end{cor}

\begin{pf}
Let $\psi\dvtx V \to\mathbb R$ satisfy $P \psi= \lambda\psi$.
By Theorem~\ref{thmpotential},
\[
\mathbb E \bigl[\dist(X_0,X_t)^2 \bigr]
\geq p_* \frac{\langle\psi
, (I-P^t) \psi\rangle}{\langle\psi, (I-P) \psi\rangle} = p_* \frac{1-\lambda^t}{1-\lambda} = p_* \bigl(1+\lambda+
\lambda^2+\cdots+\lambda^{t-1} \bigr).
\]
To complete the proof, use the inequality $\lambda^j \ge(1-t^{-1})^j
\ge1-j/t$.
\end{pf}

\begin{remark}[(Weighted graphs)]
\label{remweighted}
In particular, if $P$ is irreducible and $p_* = \min\{ P(x,y) \dvtx P(x,y)
> 0 \}$, then the
conclusion is that $\E [\dist(X_0,X_t)^2 ] \geq p_* t/2$
for $t \leq(1-\lambda)^{-1}$.
Thus if $P$ is the simple random walk on a $d$-regular graph,
the conclusion is $\E [\dist(X_0,X_t)^2 ] \geq t/(2d)$.

To see that the asymptotic dependence on $d$ is tight, one can consider
a cycle of length $n$, together with $d-2$ self loops at each vertex,
for $d \geq2$.
In this case, $\E [\dist(X_0,X_t)^2 ] \le2t/d$ for all $t
\geq0$.
\end{remark}

\begin{remark}[(After the relaxation time)]
The quantity $(1-\lambda)^{-1}$ is called the \textit{relaxation time} of
the random
walk specified by $P$, and bound \eqref{eqsumbound}
degrades after this time. It is interesting to consider
what happens between the relaxation time and the mixing time which is
always at most $O(\log|V|) (1-\lambda)^{-1}$. One might conjecture
that $\E [\dist(X_0,X_t)^2 ]$ continues to have a linear
lower bound
until the mixing time.
Toward this end, we pose the following conjecture.

\begin{conjecture}
\label{conjfinite}
There exists a constant $\epsilon_0 > 0$ such that the following holds.
For every finite, connected, $d$-regular transitive graph $G=(V,E)$ with
diameter $D$,
\[
\E \bigl[\dist(X_0,X_t)^2 \bigr] \geq
\epsilon_0 t/d
\]
for $t \leq\epsilon_0 D^2$, where $\{X_t\}$ is
the simple random walk on $G$.
\end{conjecture}
\end{remark}

We remark that the results of~\cite{DSC94} imply that the
conjecture is correct for a wide class of groups of ``moderate growth.''

\begin{cor}[(Infinite amenable graphs)]
\label{coramen}
If $G=(V,E)$ is an infinite, transitive, connected, amenable graph
with degree $d$, and $\{X_t\}$ is the simple random walk, then
\[
\E \bigl[\dist(X_0, X_t)^2 \bigr] \geq t/d
.
\]
\end{cor}

\begin{pf}
If $P$ is the transition kernel of the simple random walk,
it is a standard fact~\cite{Kesten59} that when $G$ is infinite,
connected and amenable,
the spectrum of $P$ has an accumulation point at $1$, but does not
contain $1$. Therefore, for every $\delta> 0$ and $\varepsilon> 0$,
there exists a $\delta' \in(0,\delta]$ so that, by considering the
spectral projection of $P$ onto the interval
$[1-\delta'-\varepsilon,1-\delta')$, we obtain a unit vector $\psi
\in\ell^2(V)$ (an approximate eigenvector)
for which $\langle\psi, P^t \psi\rangle\leq(1-\delta')^t$, while
$\langle\psi, P \psi\rangle\geq1-\delta'-\varepsilon$.

Plugging this into Theorem~\ref{thmpotential}, we conclude that
\[
\E \bigl[\dist(X_0, X_t)^2 \bigr] \geq
\frac{1}{d} \cdot\frac
{1-(1-\delta')^t}{\delta'+\eps} .
\]
Sending $\eps\to0$ and then $\delta\to0$ (and hence $\delta' \to
0$) yields the desired claim.
\end{pf}

\begin{question}
The preceding corollary yields a uniform lower bound of the form $\E
[\dist(X_0,X_t)^2] \geq Ct/d$ for all
$d$-regular infinite, connected, amenable graphs. In fact, one can take
$C=1$. Certainly for every $d$-regular
infinite, connected graph $G$, there exists a constant $C_G$ such that
$\E[\dist(X_0,\break X_t)^2] \geq C_G t/d$,
since in the nonamenable case $\dist(X_0, X_t)$ grows linearly with
$t$, but with some constant depending on $G$.
It is natural to ask whether one can take $C_G \geq\Omega(1)$, that
is, whether a uniform lower bound
holds without the amenability assumption. This seems closely related to
Conjecture~\ref{conjfinite}.\looseness=1
\end{question}


\subsection{Infinite amenable graphs}
\label{secamenable}

While Corollary~\ref{coramen} gives satisfactory results for
infinite, amenable graphs,
we take some time in this section to further analyze the amenable case;
in particular, the explicit
construction of
Lemma~\ref{lemgenQ} serves
as a connection between random walks and our construction of harmonic
functions in Section~\ref{secharmonic}.

The following theorem will play a role
in a number of arguments. The transitive
version is due to Soardi and Woess~\cite{SW90},
and the extension to quasi-transitive actions is
from~\cite{Salvatori92}. See also a different
proof in~\cite{BLPS99}, Theorem~3.4.

We recall that for a graph $G=(V,E)$,
we say that a group $\Gamma\leq\Aut(G)$ is \textit{quasi-transitive}
if $|\Gamma\orbit V| < \infty$, where $\Gamma\orbit V$ denotes the
set of $\Gamma$-orbits of $V$.

\begin{theorem}\label{thmamenable}
Let $G$ be a graph and $\Gamma\leq\Aut(G)$ a closed,
quasi-transitive subgroup. Then $G$ is amenable
if and only if $\Gamma$ is amenable and unimodular.
\end{theorem}

We begin with the following general construction.
Gromov~\cite{Gromov03}, Sec-\break tions~3.6--3.7, uses a similar
analysis in the setting of the continuous
heat flow on manifolds (see, in particular,
Remark 3.7(B) in~\cite{Gromov03}).
We remark that, in this setting,
the result itself follows rather directly from spectral projection
as in the proof of Corollary~\ref{coramen}.
We present the following proof because it is quite elementary and explicit,
and directly relates our harmonic function construction (Section \ref
{secharmonic})
to random walks.

\begin{lemma}\label{lemgenQ}
Let $\mathcal{H}$ be a Hilbert space, and let $Q \dvtx \mathcal{H} \to
\mathcal{H}$
be a self-adjoint linear operator which is contractive, that is, with
$\| Q\| _{\mathcal{H} \to\mathcal{H}}\leq1$.
Suppose that for some $\theta\in(0,\frac12)$,
there exists an $f \in\mathcal{H}$ which satisfies $\| f\| _{\mathcal
{H}}=1$, $\| Qf-f\| _{\mathcal{H}} \leq\theta$,
and
%
\begin{equation}
\label{eqlimk} \lim_{k \to\infty} \frac{1}{k} \sum
_{i=0}^{k-1} \bigl\langle Q^i f,f \bigr
\rangle= 0.
\end{equation}
Then there exists an element $\varphi\in\mathcal{H}$
with
\[
\frac{\| (I-Q)\varphi\| ^2_{\mathcal{H}}}{\langle\varphi,
(I-Q)\varphi\rangle_{\mathcal{H}}} \leq32 \theta.
\]
\end{lemma}

Before venturing into the proof, we mention a natural approach for a
special case. Consider the situation where $\Gamma$
is a finitely-generated, amenable group, and $G$ is the corresponding
Cayley graph.
Let $P$ be the transition kernel of the corresponding random walk, and
let $\{F_k\}$ be a F{\o}lner sequence in $G$
which satisfies
$\langle\1_{F_k}, P^i \1_{F_k} \rangle\geq\frac12 |F_k|$ for
$i=0,1,2,\ldots,k$.

In that case, one might choose to put
%
\begin{equation}
\label{eqmightchoose} \psi_k = \sum_{i=0}^{\infty}
P^i \1_{F_k} ,
\end{equation}
where $\1_{F_k}$ is the characteristic function of $F_k$.
Assume, for the moment, that $\psi_k \in\ell^2(V)$.
In this case,
$
(I-P) \psi_k = \1_{F_k},
$
so
\[
\bigl\| (I-P) \psi_k\bigr\| ^2 = |F_k|,
\]
while
\[
\bigl\langle\psi_k, (I-P) \psi_k \bigr\rangle= \sum
_{i=0}^{\infty} \bigl\langle \1_{F_k},
P^i \1_{F_k} \bigr\rangle\geq\frac{k}{2}
|F_k| ,
\]
by our assumption on $\{F_k\}$.
This implies that
\[
\lim_{k \to\infty} \frac{\| (I-P) \psi_k\| ^2}{\langle\psi_k,
(I-P)\psi_k\rangle} \to0,
\]
yielding an analog to the conclusion of Lemma~\ref{lemgenQ}.

The only remaining issue is whether $\psi_k \in\ell^2(V)$, which requires
some assumptions on the group $\Gamma$. To get around this, we truncate
the sum in \eqref{eqmightchoose}. The somewhat delicate issue that arises
is where to truncate. The following proof shows that a good truncation
point always exists.

%

\begin{pf*}{Proof of Lemma~\ref{lemgenQ}}
Given $f \in\mathcal{H}$ and $k \in\mathbb N$, we define $\varphi_{k} \in\mathcal{H}$ by
\[
\varphi_{k} = \sum_{i=0}^{k-1}
Q^i f.
\]
First, using $(I-Q)\varphi_{k} = (I-Q^k) f$ and the fact that $Q$ is a
contraction, we have
%
\begin{equation}
\label{eqlapest} \bigl\| (I-Q)\varphi_{k}\bigr\| _{\mathcal H}^2
\leq4 \| f\| _{\mathcal H}^2.
\end{equation}

On the other hand,
\begin{eqnarray*}
\bigl\langle\varphi_{k}, (I-Q)\varphi_{k} \bigr
\rangle_{\mathcal H} &=& \bigl\langle\varphi_{k},
\bigl(I-Q^k \bigr) f \bigr\rangle_{\mathcal H}
\\
&=& \Biggl\langle \bigl(I-Q^k \bigr) \sum
_{i=0}^{k-1} Q^i f, f \Biggr
\rangle_{\mathcal H}
\\
&=& \langle2 \varphi_{k}-\varphi_{2k}, f
\rangle_{\mathcal H},
\end{eqnarray*}
where in the second line we have used the fact that $I-Q^k$ is self-adjoint.
Combining this with \eqref{eqlapest} yields
%
\begin{equation}
\label{eqratio1} \frac{\| (I-Q)\varphi_{k}\| _{\mathcal H}^2}{\langle\varphi_{k},
(I-Q)\varphi_{k}\rangle_{\mathcal H}} \leq\frac{4 \| f\| _{\mathcal
H}^2}{\langle2 \varphi_{k}-\varphi_{2k},f\rangle_{\mathcal H}}.
\end{equation}
%
The following claim will conclude the proof.

\begin{claim*} There exists a $k \in\mathbb N$
such that
%
\begin{equation}
\label{eqwantit} \langle2 \varphi_{k} - \varphi_{2k}, f
\rangle_{\mathcal H} \geq \frac{1}{8\theta}.\vadjust{\goodbreak}
\end{equation}
\end{claim*}
It remains to prove the claim. By assumption, $f$ satisfies $\| f\| _{\mathcal H}=1$, and
$\| Qf-f\| _{\mathcal H} \leq\theta$. Since $Q$ is a contraction, we
have $\| Q^j f - Q^{j-1} f\| _{\mathcal H} \leq\theta$
for every $j \geq1$, and thus by the triangle inequality, $\| Q^j f -
f\| _{\mathcal H} \leq j \theta$ for every $j \geq1$.
It follows by Cauchy--Schwarz that $\langle f, (I-Q^j) f \rangle_{\mathcal H} \leq j\theta$, therefore
\[
\label{eqinnerprod} \bigl\langle f, Q^j f \bigr\rangle_{\mathcal H}
\geq1-j \theta.
\]
Thus for every $j \geq1$,
\[
\langle\varphi_{2^j}, f \rangle_{\mathcal H} \geq2^j
\bigl(1-2^j \theta \bigr).
\]

Fix $\ell\in\mathbb N$ so that $2^{\ell} \theta\leq\frac12 \leq
2^{\ell+1} \theta$,
yielding
%
\begin{equation}
\label{eqM} \langle\varphi_{2^{\ell}}, f \rangle_{\mathcal H} \geq
\frac
{1}{8\theta}.
\end{equation}

Now, let $a_m = \langle\varphi_{2^m}, f \rangle_{\mathcal H}$, and
write, for some $N \geq1$,
\[
a_{\ell} - \frac{a_{N}}{2^{N-\ell}} = \sum_{m=\ell}^{N-1}
\frac{2
a_m - a_{m+1}}{2^{m-\ell+1}}.
\]
By \eqref{eqlimk}, we have
%
\begin{equation}
\label{eqlimzero} \lim_{N \to\infty} \frac{a_N}{2^N} = 0.
\end{equation}
Using \eqref{eqM} and
taking $N \to\infty$ on both sides above yields
\[
\frac{1}{8\theta} \leq a_{\ell} = \sum_{m=\ell}^{\infty}
\frac{2
a_m - a_{m+1}}{2^{m-\ell+1}}.
\]
Since $\sum_{m=\ell}^{\infty} \frac{1}{2^{m-\ell+1}}=1$, there
must exist some $m \geq\ell$ with $2 a_m - a_{m+1} \geq\frac
{1}{8\theta}$.
This establishes claim~(\ref{eqwantit}) for $k=2^m$, and in view of
(\ref{eqratio1}), completes the proof of the lemma.
\end{pf*}

We now arrive at the following corollaries.
Recall that if $P$ is transient or null-recurrent,
then we have the pointwise limit,
%
\begin{equation}
\label{eqpwlimit} P^i f \to0\qquad \mbox{for every } f \in
\ell^2(V).
\end{equation}
(This is usually proved for finitely supported $f$; see, e.g.,
\cite{GS92}, Theorem 6.4.17 or~\cite{LPW09}, Theorem 21.17.
The general case follows by approximation using the contraction
property of $P$.)

\begin{cor}\label{corinfpsi}
If $V$ is infinite, $P$ satisfies \eqref{eqpwlimit},
and $\Gamma\leq\Aut(P)$ is a closed, amenable, unimodular subgroup,
which acts transitively on $V$, then
%
\begin{equation}
\label{eqinfpsi} \inf_{\varphi\in\ell^2(V)} \frac{\| (I-P) \varphi\| ^2}{\langle
\varphi, (I-P) \varphi\rangle} = 0.
\end{equation}
\end{cor}

\begin{pf}
This follows from Lemma~\ref{lemgenQ} using
the fact that, under the stated assumptions,
for every $\theta> 0$, there exists an $f \in\ell^2(V)$
with $\| f\| =1$ and $\| Pf-f\|  \leq\theta$.

This is a standard fact that can be proved, as in
\cite{Woess00}, Theorem 12.10. In general,
for every $\theta> 0$,
one considers, for some $\varepsilon= \varepsilon(\theta) > 0$, the
graph $G_{\varepsilon}$ with vertices $V$
and an edge $\{x,y\}$ whenever $P(x,y) \geq\varepsilon$.
Since $\Gamma\leq\Aut(G_\varepsilon)$, Theorem~\ref{thmamenable}
implies that
$G_{\varepsilon}$ is amenable, and then one can take $f=|F|^{-1/2} \1_{F}$ to be the
(normalized) indicator of a suitable F{\o}lner set $F \subseteq V$ in
$G_\varepsilon$.
The idea is that for $\varepsilon> 0$ chosen small enough,
$\| P\1_F-\1_F\| ^2$ is close to the size of the outer vertex boundary of
$F$ in $G_\varepsilon$.
\end{pf}

The following is an immediate consequence of Theorem~\ref{thmpotential}
combined with the preceding result.

\begin{cor}[(The amenable case)]\label{corinfinite-drift}
Under the assumptions of Theorem~\ref{thmpotential}, the following holds.
If $V$ is a countably infinite index set,
$P$ satisfies~\eqref{eqpwlimit},
and $\Gamma\leq\Aut(P)$ is a closed, amenable,
unimodular subgroup which acts transitively on $V$, then
%
\begin{equation}
\label{eqinfdrift2} \E \bigl[\dist(X_0,X_t)^2
\bigr] \geq p_* t.
\end{equation}
%
\end{cor}


%

\begin{cor}[(The nearly amenable case, for small times)]
Under the assumptions of Theorem~\ref{thmpotential}, the following holds.
If $\rho=\rho(P)$ is the spectral radius of $P$, then
for all times $t \leq(32 (1-\rho))^{-1}$,
\[
\mathbb E \bigl[\dist(X_0,X_t)^2 \bigr]
\geq\frac{p_* t}{2}.
\]
\end{cor}

\begin{pf}
Since $P$ is self-adjoint and positive, by standard
variational principles, we have
$\rho=\| P\| _{2 \rightarrow2} = \sup_{\| f\| =1} \langle Pf,f \rangle.
$
It follows that
\[
\inf_{\| f\| =1} \| f-Pf\| ^2 \leq \Bigl(\inf_{\| f\| =1} 1 +
\rho^2 - 2 \langle f, Pf \rangle \Bigr) = (1-\rho)^2.
\]
Combining this with Lemma~\ref{lemgenQ} yields the claimed result.
\end{pf}

Compare the preceding bound with the finite case (Corollary~\ref{corfinite}).

\begin{remark}[(Asymptotic rate of escape)]
\label{remasymptotic}
The constant $p_*$ in \eqref{eqinfdrift} is not tight.
To do slightly better, one can argue as follows.
Let $\Psi: V \to L^2(\Gamma,\mu)$ be as
in the proof of Theorem~\ref{thmpotential}.
Fix $x,y \in V$ with $L = \dist(x,y)$, and let $x = v_0, v_1, \ldots,
v_L = y$
be a shortest path from $x$ to $y$ in $G$.
In this case, the triangle inequality yields
\begin{eqnarray*}
 2 \bigl\| \Psi(x)-\Psi(y)\bigr\|
&\leq&\bigl\| \Psi(v_0)-\Psi(v_1)\bigr\| \\
&&\quad{} +
\sum_{i=1}^{L-1} \bigl(\bigl\| \Psi(v_{i-1})-
\Psi(v_{i})\bigr\| +\bigl\| \Psi(v_i)-\Psi (v_{i+1})\bigr\|  \bigr)
\\
&&\quad{}+ \bigl\| \Psi(v_{L-1})-\Psi(v_L)\bigr\| .
\end{eqnarray*}
But for every $i \in\{1,2,\ldots,L-1\}$, there are two terms
involving $v_i$, and for such~$i$, we can bound
\[
\bigl\| \Psi(v_i)-\Psi(v_{i-1})\bigr\| ^2 + \bigl\|
\Psi(v_i)-\Psi(v_{i+1})\bigr\| ^2 \leq
\frac{2 \langle\psi, (I-P)\psi\rangle}{p_*}
\]
as in \eqref{eqlastbound}. In this way, we gain a factor of 2 for
such terms.
Letting $\alpha$ denote the right-hand side of the preceding
inequality, we have
\[
\bigl\| \Psi(x)-\Psi(y)\bigr\|  \leq\sqrt{\alpha} \biggl(1+ \frac{L-1}{\sqrt {2}} \biggr) \leq
\sqrt{\frac{\alpha}{2}} (L+1).
\]
Thus for all $x,y \in V$, we have $\| \Psi(x)-\Psi(y)\| ^2 \leq[\dist
(x,y)+1]^2 \frac{\langle\psi, (I-P)\psi\rangle}{p_*}$.
Plugging this improvement into the proof of Theorem \ref
{thmpotential} yields
%
\begin{equation}
\label{eqoptimal} \mathbb E \bigl[ \bigl(\dist(X_0,X_t)+1
\bigr)^2 \bigr] \geq2 p_* \frac{\langle
\psi, (I-P^t) \psi\rangle}{\langle\psi, (I-P) \psi\rangle},
\end{equation}
which is asymptotically tight since, on the one hand, the simple random
walk on~$\mathbb Z$ satisfies $\E [\dist(X_0,X_t)^2 ] =t$,
while plugging \eqref{eqoptimal} into Corollary~\ref
{corinfinite-drift} yields
$\E [(\dist(X_0,X_t)+1)^2 ] \geq t$.

The dependence on $p_*$ is easily seen to be tight for the simple
random walk on~$\mathbb Z$ with a $1-2p_*$
holding probability added to every vertex, as in Remark~\ref{remweighted}.
\end{remark}

\section{Equivariant harmonic maps}
\label{secharmonic}

Let $V$ be a countably infinite index set,
and let $\{P(x,y)\}_{x,y \in V}$ be a
stochastic, symmetric matrix.
If $\mathcal H$ is a Hilbert space,
a~mapping $\Psi: V \to\mathcal H$ is called \textit{$P$-harmonic}
if, for all $x \in V$,
\[
\Psi(x) = \sum_{y \in V} P(x,y) \Psi(y).
\]
Suppose furthermore that we have
a group $\Gamma$ acting on $V$.
We say that
$\Psi$ is \textit{$\Gamma$-equivariant} if there exists an affine
isometric action $\pi$ of $\Gamma$
on $\mathcal H$, such that for every $g \in\Gamma$,
$\pi(g) \Psi(x) = \Psi(g x)$ for all $x \in V$.
If we wish to emphasize the particular action $\pi$, we will say that
$\Psi$ is
\textit{$\Gamma$-equivariant with respect to $\pi$}.

We remark in passing
that there do exist amenable groups admitting a nonunimodular action;
see, for example,~\cite{PPS06}, Example 2.2.

\begin{theorem}\label{thmharmonic}
For $P$ as above, let $\Gamma\leq\Aut(P)$ be a closed, amenable,
unimodular subgroup
which acts transitively on $V$.
Suppose there exists a connected graph $G=(V,E)$ on which
$\Gamma$ acts by automorphisms,\vadjust{\goodbreak} and that for~$x \in V$,
%
\begin{equation}
\label{eqbddstep} \sum_{y \in V} P(x,y)
\dist(x,y)^2 < \infty,
\end{equation}
where $\dist$ is the path metric on $G$.
Suppose also that
\[
p_* = \min \bigl\{ P(x,y) \dvtx \{x,y\} \in E \bigr\} > 0.
\]
Then there exists a Hilbert space $\mathcal H$,
and a nonconstant $\Gamma$-equivariant $P$-harmonic mapping from $V$
into $\mathcal H$.
\end{theorem}

%
%

\begin{pf}
It is a standard result that since $G$ is connected, $P$ satisfies~\eqref{eqpwlimit}.
Let $\{\psi_j\} \subseteq\ell^2(V)$ be a sequence of functions satisfying
%
\begin{equation}
\label{eqratio} \frac{ \langle(I-P) \psi_j, (I-P) \psi_j \rangle}{
\langle\psi_j, (I-P) \psi_j  \rangle} \to0.
\end{equation}
The existence of such a sequence is the content of Corollary~\ref{corinfpsi}.

Define
$\Psi_j \dvtx V \to L^2(\Gamma,\mu)$ by
\[
\Psi_j(x) \dvtx \sigma\mapsto\frac{\psi_j(\sigma^{-1} x)}{\sqrt{2
\langle\psi_j, (I-P)\psi_j\rangle}}.
\]
%

Since $\Gamma$ is unimodular, we can choose the Haar measure $\mu$
on $\Gamma$ to be normalized $\mu(\Gamma_x)=1$ for all $x \in V$,
where $\Gamma_x$ is the stabilizer subgroup of $x$.

%
%
%
%
Now, observe that for every $x \in V$,
%
\begin{eqnarray}
\label{eqlipest}
&&\sum_{y \in V} P(x,y) \bigl\|
\Psi_j(x)- \Psi_j(y)\bigr\| _{L^2(\Gamma,\mu)}^2\nonumber\\
 &&\qquad=
\frac{ \sum_{y \in V} P(x,y) \int|\psi_j(\sigma^{-1} x)-\psi_j(\sigma^{-1} y)|^2 \,d\mu(\sigma)}{ 2
\langle\psi_j, (I-P) \psi_j\rangle}
\nonumber
\\[-8pt]
\\[-8pt]
\nonumber
&&\qquad= \mu(\Gamma_x) \frac{ \sum_{u,y \in V} P(u,y) |\psi_j(u)-\psi_j(y)|^2}{2 \langle\psi_j, (I-P) \psi_j \rangle}
\\
&&\qquad= 1.\nonumber
\end{eqnarray}

Next, for every $x \in V$, we have
\begin{eqnarray*}
&&\biggl\llVert \Psi_j(x)-\sum_{y \in V}
P(x,y) \Psi_j(y) \biggr\rrVert_{L^2(\Gamma,\mu)}^2\\
&&\qquad=
\frac{ \int |\psi_j(\sigma^{-1} x)-\sum_{y
\in V} P(x,y) \psi_j(\sigma^{-1} y) |^2 \,d\mu(\sigma)}{2 \langle
\psi_j, (I-P)\psi_j\rangle}
\\
&&\qquad= \mu(\Gamma_x) \frac{ \sum_{u \in V} \llvert \psi_j(u)-\sum_{y \in V} P(u,y) \psi_j(y)\rrvert^2}{2 \langle\psi_j,
(I-P) \psi_j\rangle}
\\
&&\qquad=
\frac{ \langle(I-P) \psi_j, (I-P) \psi_j \rangle
}{2 \langle\psi_j, (I-P) \psi_j  \rangle}\ . 
\end{eqnarray*}
In particular, from \eqref{eqratio},
%
\begin{equation}
\label{eqharmlim} \lim_{j \to\infty} \biggl\llVert \Psi_j(x)-\sum
_{y \in V} P(x,y) \Psi_j(y) \biggr
\rrVert_{L^2(\Gamma,\mu)}^2 = 0,
\end{equation}
where the limit is uniform in $x \in V$.

%
%


Define a unitary action $\pi_0$ of $\Gamma$ on $L^2(\Gamma,\mu)$ as follows:
For $\gamma\in\Gamma, h \in L^2(\Gamma,\mu)$,
$[\pi_0(\gamma) h](\sigma) = h(\gamma^{-1} \sigma)$
for all $\sigma\in\Gamma$.
Notice that each $\Psi_j$ is $\Gamma$-equivariant since for $\gamma
\in\Gamma, x \in V$, we have
\[
\bigl(\pi_0(\gamma) \bigl[\Psi_j(x) \bigr] \bigr) (
\sigma) = \bigl[\Psi_j(x) \bigr] \bigl(\gamma^{-1} \sigma
\bigr) = \frac{\psi_j(\sigma^{-1} \gamma x)}{\sqrt{2 \langle\psi_j, (I-P)\psi_j\rangle}} = \bigl[\Psi_j(\gamma x) \bigr](\sigma).
\]

We state the next lemma in slightly more generality than we need
presently, since we will use it also in Section~\ref{secwithoutT}.

\begin{lemma}\label{lemgencon}
Suppose that $\mathfrak{H}$ is a Hilbert space, $\Gamma$ is a group,
and $\pi_0$ is an affine
isometric action of $\Gamma$ on $\mathfrak{H}$. Let $(V,\mathsf
{dist})$ be a countable metric space,
and consider a sequence of functions $\{\Psi_j \dvtx V \to\mathfrak{H}\}_{j=1}^{\infty}$,
where the $\Psi_j$'s are uniformly Lipschitz and $\Gamma$-equivariant
with respect to $\pi_0$.
Then there is a sequence of affine isometries $T_j \dvtx \mathfrak{H} \to
\mathfrak{H}$
and a subsequence $\{\alpha_j\}$ such that $T_{\alpha_j} \Psi_{\alpha_j}$ converges
pointwise to a map $\widehat{\Psi} \dvtx V \to\mathfrak{H}$ which is
$\Gamma$-equivariant
with respect to an affine isometric action $\pi$.
\end{lemma}

Before proving the lemma, let us see that it can be used
to finish the proof of Theorem~\ref{thmharmonic}.
Using \eqref{eqlipest}, one observes that for all $j \in\mathbb N$,
the map $\Psi_j$ is $\sqrt{1/p_*}$-Lipschitz
on $(V,\mathsf{dist})$.
Thus we are in position to apply the preceding lemma and arrive at a map
$\widehat{\Psi}\dvtx V \to L^2(\Gamma,\mu)$ which is $\Gamma$-equivariant
with respect to an affine isometric action.

From \eqref{eqharmlim}, we see that
$\widehat{\Psi}$ is $P$-harmonic.
Furthermore, since the ${\Psi}_j$'s are
uniformly Lipschitz, and we have the estimate \eqref{eqbddstep},
we see that \eqref{eqlipest} holds for $\widehat{\Psi}$ as well,
showing that $\widehat{\Psi}$ is nonconstant,
and completing the proof.

\begin{pf*}{Proof of Lemma~\ref{lemgencon}}
Arbitrarily order
the points of $V = \{x_1, x_2, \ldots\}$ and
fix a sequence
of subspaces $\{W_j\}_{j=1}^{\infty}$ of $\mathfrak{H}$
with $W_j \subseteq W_{j+1}$ for each $j=1,2,\ldots,$ and $\dim(W_j)=j$.
For each such $j$, define an affine isometry $T_j \dvtx \mathfrak{H} \to
\mathfrak{H}$ which
satisfies $T_j \Psi_j (x_1) = 0$ and, for every $r=1,2,\ldots, j$,
$\{T_j \Psi_j(x_k)\}_{k=1}^r \subseteq W_r$.
Put $\widehat{\Psi}_j = T_j \Psi_j$, and define an affine isometric
action $\pi_j$ of $\Gamma$ on $\mathfrak{H}$
by $\pi_j = T_j \pi_0 T_j^{-1}$. It is straightforward to check that
each $\widehat{\Psi}_j$
is $\Gamma$-equivariant with respect $\pi_j$.

Since the maps $\{\Psi_j\}$ are uniformly Lipschitz, the same holds
for the family $\{\widehat{\Psi}_j\}$.
We now pass to a subsequence
$\{\alpha_j\}$ along which $\widehat{\Psi}_{\alpha_j}(x)$
converges pointwise for every $x \in V$.
To see that this is possible, notice that by construction, for every
fixed $x \in V$,
there is a finite-dimensional subspace $W \subseteq\mathfrak{H}$ such
that $\widehat{\Psi}_j(x) \subseteq W$
for every $j \in\mathbb N$. Hence by the uniform Lipschitz property of
$\widehat{\Psi}_j$, the sequence
$\{\widehat{\Psi}_j(x)\}_{j=1}^{\infty}$ lies in a compact set.

We are thus left to show that $\widehat{\Psi}$ is $\Gamma$-equivariant.
Toward this end, we define an action $\pi$ of $\Gamma$ on $\mathfrak
{H}$ as follows:
On the image of $\widehat{\Psi}$, set
$\pi(\gamma) \widehat{\Psi}(x) = \widehat{\Psi}(\gamma x)$.
For $g \in\mathfrak{H}$ lying in the orthogonal complement of the span
of $\{\widehat{\Psi}(x)\}_{x \in V}$, we
put $\pi(\gamma) g = \pi(\gamma) 0$, and then
extend $\pi(\gamma)$ affine linearly to the whole space.
To see that such an affine linear
extension exists,
observe that
\[
\pi(\gamma) \widehat{\Psi}(x) = \widehat{\Psi}(\gamma x) = \lim_{j\to\infty}
\pi_{\alpha_j}(\gamma) \widehat{\Psi}_{\alpha_j}(x).
\]
From this expression, it follows immediately that $\pi$ acts by affine
isometries,
since each $\pi_{\alpha_j}$ does.
Thus $\widehat{\Psi}$ is $\Gamma$-equivariant with respect to $\pi
$, completing our construction.
\end{pf*}
\noqed\end{pf}

\begin{remark}
Note that, in the case where $P$ is simply the kernel of the
simple random walk on the Cayley graph of a finitely-generated
amenable group, one
can take the Hilbert space $\mathcal H$ in Theorem~\ref{thmharmonic}
to be simply $\ell^2(V)$.
\end{remark}

\subsection{Quasi-transitive graphs}

Only for the present section, we allow $P$ to be a nonsymmetric kernel
on the state space $V$.
We recall that $\Gamma$ is said to \textit{act quasi-transitively on a
set $V$} if
$|\Gamma\orbit V| < \infty$, where $\Gamma\orbit V$ denotes
the set of $\Gamma$-orbits of $V$.
We prove an analog of Theorem~\ref{thmharmonic} in the
quasi-transitive setting,
under the assumption that the kernel $P$ is reversible.

\begin{cor}[(Quasi-transitive actions)]
\label{corquasi}
Let $\Gamma\leq\Aut(P)$ be a closed, amenable, unimodular subgroup
which acts \textit{quasi-transitively} on $V$.
Suppose also that $P$ is the kernel of a reversible Markov chain,
and there exists a connected graph $G=(V,E)$ on which
$\Gamma$ acts by automorphisms, and that for $x \in V$,
%
\begin{equation}
\label{eqbddstep3} \sum_{y \in V} P(x,y)
\dist(x,y)^2 < \infty,
\end{equation}
where $\dist$ is the path metric on $G$.
Suppose also that
\[
p_* = \min \bigl\{ P(x,y) \dvtx \{x,y\} \in E \bigr\} > 0.
\]
Then there exists a Hilbert space $\mathcal H$,
and a nonconstant $\Gamma$-equivariant $P$-harmonic mapping from $V$
into $\mathcal H$.
\end{cor}

\begin{pf}
Let $x_0, x_1, \ldots, x_L \in V$ be a complete set of representatives
of the orbits of $\Gamma$.
Let $V_0 = \Gamma x_0$, and let $P_0$ be the induced\vadjust{\goodbreak} transition kernel
of the $P$-random walk watched on $V_0$,
that is, $P_0(x,y)~=~\pr[X_{\tau} = y  \mid  X_0 = x]$ for $x,y
\in V_0$, where
$\tau= \min\{ t \geq1 \dvtx X_t \in V_0 \}$.
Since $P$ is reversible and $\Gamma$ acts transitively on $V_0$, we
see that $P_0$ is symmetric.

Letting $D = \max_{i \neq j} \dist(x_i, x_j)$, we define a new graph
$G_0=(V_0,E_0)$
by having an edge $\{x,y\} \in E_0$ whenever:
\begin{longlist}[(1)]
\item[(1)] $\{x,y\} \in E$ and $x,y \in V_0$ or
\item[(2)] there exists a path $x = v_0, v_1, \ldots, v_k = y$ in $G$
with $v_1, \ldots, v_{k-1} \notin V_0$ and $k \leq2D$.
\end{longlist}
Let $\dist_0$
denote the path metric on $G_0$.
It is clear that $\Gamma$ acts on $G_0$ by automorphisms, and also that
$p_*(G_0) = \min\{ P_0(x,y) \dvtx \{x,y\} \in E_0 \} \geq(p_*)^{2D} > 0$.

Now, since every point $x \notin V_0$ has $\dist(x,V_0) \leq D$, we
see that actually
$\dist(x,y) \eqsim \dist_0(x,y)$ for all $x,y \in V_0$ (up to a
multiplicative
constant depending on $D$). Furthermore, this implies that
for any $x \in V$ there exists $y \in V_0$ with $\sum_{i=0}^{D}
P^i(x,y) \geq(p_*)^D$, and
hence \eqref{eqbddstep3} implies that for every $x \in V_0$,
\[
\sum_{y \in V_0} P_0(x,y)
\dist(x,y)^2 < \infty,
\]
since
number of $P$-steps taking before returning to $V_0$ is dominated
by a geometric random variable.
This implies the same for $\dist_0$.

Thus we can apply Theorem~\ref{thmharmonic} to obtain a Hilbert space
$\mathcal H$ and a
nonconstant $\Gamma$-equivariant $P_0$-harmonic map $\Psi_0 \dvtx V_0
\to\mathcal H$.
We extend this to a mapping $\Psi: V \to\mathcal H$ by defining
$\Psi(x) = \E [ \Psi_0(W_0(x)) ]$ where $W_0(x)$ is the
first element
of $V_0$ encountered in the $P$-random walk started at $x$.
Note that $\Psi|_{V_0} = \Psi_0$, and $\Psi$ is again $\Gamma$-equivariant.
To finish the proof, it suffices to check that $\Psi$ is $P$-harmonic.

From the definition of $\Psi$, this is immediately clear for $x \notin
V_0$. Since $\Psi_0$
is $P_0$-harmonic, it suffices to check that for $x \in V_0$,
\[
\sum_{y \in V} P(x,y) \Psi(y) = \sum
_{y \in V_0} P_0(x,y) \Psi_0(y),
\]
but both sides are precisely $\E [\Psi_0(W_0(Z)) ]$,
where $Z$ is the random vertex arising from one step of the $P$-walk
started at $x$.
\end{pf}

\begin{cor}[(Harmonic functions on quasi-transitive graphs)]
\label{corqtharm}
If $G=(V,E)$ is an infinite, connected, amenable graph,
and $\Gamma\leq\Aut(G)$ is a quasi-transitive subgroup,
then $G$ admits a nonconstant $\Gamma$-equivariant harmonic mapping
into some Hilbert space.
\end{cor}

Now let $G=(V,E)$ be an infinite, connected, quasi-transitive, amenable graph.
The preceding construction of harmonic functions also gives escape
lower bounds
for the random walk on $G$.
By Theorem~\ref{thmamenable}, when $G$ is amenable, $\Gamma=\Aut
(G)$ is amenable and unimodular.
Let $R \subseteq V$ be a complete set of representatives from $\Gamma
\orbit V$.
Let $\mu$ be the Haar measure on $\Gamma$.
For $r \in R$, let $\mu_r = \mu(\Gamma_r)$, and
normalize $\mu$ so that
$\sum_{r \in R} \deg(r)/\mu_r = 1$.

\begin{cor}[(Random walks on quasi-transitive graphs)]
Let $\dist$ be the path metric on $G$,
and let $X_0$ have the distribution
$\pr[X_0 = r] = \deg(r)/\mu_r$
for $r \in R$. Then
%
\begin{equation}
\E \bigl[\dist(X_0, X_t)^2 \bigr] \geq
\frac{t}{\max \{ \mu_r \dvtx r \in R  \}},
\end{equation}
where $\{X_t\}$ denotes the simple random walk on $G$.
\end{cor}

\begin{pf}
Let $\Psi: V \to\mathcal H$ be the harmonic map guaranteed by
Corollary~\ref{corqtharm}
normalized so that
%
\begin{equation}
\label{eqqtgradient} \sum_{r \in R} \frac{1}{\mu_r}
\sum_{x : \{x,r\} \in E} \bigl\| \Psi (x)-\Psi(r)\bigr\| ^2 = 1.
\end{equation}
We have
\[
\| \Psi\| _{\Lip} \leq\max_{r \in R} \sqrt{\mu_r}.
\]

For every $r, \hat r \in R$, the mass transport principle~\cite{BLPS99}, Corollary
3.5,
implies that
\[
\frac{1}{\mu_r} \# \bigl\{ x \in\Gamma\hat r \dvtx \{r,x\} \in E \bigr\} =
\frac{1}{\mu_{\hat r}} \# \bigl\{ x \in\Gamma r \dvtx \{\hat r,x\} \in E \bigr\}.
\]

Thus if we use
the notation $[x]$ to denote the unique $r \in R$ such that $x \in
\Gamma r$,
then $[X_i]$ and $[X_0]$ are identically distributed for every $i \geq0$.
(This is also a special case of~\cite{LS99}, Theorem 3.1.)
It follows that
\begin{eqnarray*}
\| \Psi\| _{\Lip}^2 \E \bigl[\dist(X_0,
X_t)^2 \bigr] &\ge& \E \bigl\| \Psi(X_0)-
\Psi(X_t)\bigr\| ^2
\\
&=& \sum_{i=0}^{t-1} \E \bigl\|
\Psi(X_{i+1})-\Psi(X_{i})\bigr\| ^2
\\
&=& \sum_{i=0}^{t-1} \E\frac{1}{\deg(X_{i})}
\sum_{x : \{x,X_i\} \in
E} \bigl\| \Psi(x)-\Psi(X_i)
\bigr\| ^2
\\
&=& t \cdot\sum_{r \in R} \frac{\deg(r)}{\mu_r}
\frac{1}{\deg(r)} \sum_{x : \{x,r\} \in E} \bigl\| \Psi(x)-\Psi(r)
\bigr\| ^2
\\
&=& t,
\end{eqnarray*}
where in the second line we have used the fact that $\{\Psi(X_t)\}$ is
a martingale,
in the fourth line we have used equivariance of $\Psi$ and the fact
that each $[X_i]$ has the same distribution, and in the final line we
have used
\eqref{eqqtgradient}.
\end{pf}

%
%

\subsection{Groups without property \textup{(T)}}
\label{secwithoutT}

We now state a version of Theorem~\ref{thmharmonic} that
applies to the simple random walk on Cayley graphs
of groups without property~(T).
(We refer to~\cite{BHV08} for a thorough discussion of Kazhdan's
property (T).)
To this end, let $\Gamma$ be a finitely-generated group,
with finite, symmetric generating set $S \subseteq\Gamma$.
Let $P$ be the transition kernel of the simple random walk
on $\Gamma$ (with steps from $S$).

\begin{theorem}
\label{thmpropTphi}
Under the preceding assumptions, if $\Gamma$ does not have property~\textup{(T)},
there
exists a Hilbert space $\mathcal{H}$ and a unitary action $\pi$ of
$\Gamma$ on $\mathcal{H}$ such that
\[
\inf_{\varphi\in\mathcal{H}} \frac{\| (I-P_{\dag})\varphi\| ^2_{\mathcal{H}}}{\langle\varphi,
(I-P_{\dag})\varphi\rangle_{\mathcal{H}}} = 0,
\]
where $P_{\dag} \dvtx \mathcal H \to\mathcal H$ is defined by
%
\begin{equation}
\label{eqPaction} P_{\dag} f = \frac{1}{|S|} \sum
_{\gamma\in S} \pi(\gamma) f.
\end{equation}
\end{theorem}

\begin{pf}
Since $\Gamma$ does not have property~(T),
it admits a unitary action $\pi$ on some Hilbert space $\mathcal{H}$ without
fixed points such that we can find, for every $\theta> 0$,
an $f \in\mathcal{H}$ with $\| f\| _{\mathcal H} = 1$ and $\| P_{\dag}
f - f \| _{\mathcal H} \leq\theta$.
Now, $P_{\dag}$ is self-adjoint and contractive by construction, thus
to apply Lemma~\ref{lemgenQ} (with $Q = P_{\dag}$) and reach our
desired conclusion,
we are left to show that
$\lim\frac{1}{k} \sum_{i=0}^{k-1} \langle P_{\dag}^i f, f \rangle= 0$.

Fix some nonzero $f \in\mathcal{H}$, and let
$\varphi_k = \frac{1}{k} \sum_{i=0}^{k-1} P_{\dag}^i f$. If
%
\begin{equation}
\label{eqnonzerolim} \lim_{k \to\infty} \frac{1}{k} \sum
_{i=0}^{k-1} \bigl\langle P_{\dag
}^i
f, f \bigr\rangle\neq0,
\end{equation}
then there exists a subsequence $\{k_{\alpha}\}$ and a nonzero
$\varphi\in\mathcal{H}$
such that $\varphi$ is a weak limit of $\{\varphi_{k_\alpha}\}$.

But in this case,
we claim that
%
\begin{equation}
\label{eqfp} P_{\dag} \varphi= \varphi,
\end{equation}
since for any $g \in\mathcal{H}$, we have
\begin{eqnarray*}
\langle P_{\dag} \varphi, g \rangle_{\mathcal H} &=&
\lim_{\alpha\to\infty} \Biggl\langle\frac{1}{k_{\alpha}} \sum
_{i=0}^{k_{\alpha}-1} P_{\dag}^{i+1} f, g
\Biggr\rangle_{\mathcal
H}
\\
&=& \lim_{\alpha\to\infty} \Biggl\langle\frac{1}{k_{\alpha}} \sum
_{i=0}^{k_{\alpha}-1} P_{\dag}^i f, g
\Biggr\rangle_{\mathcal H}
\\
&=& \langle\varphi, g \rangle_{\mathcal H},
\end{eqnarray*}
where we have used
the fact that $\lim_{\alpha\to\infty} \frac{1}{k_{\alpha}}
(P_{\dag}^{k_{\alpha}} f - f) = 0,$
since $P_{\dag}$ is contractive.
On the other hand, if \eqref{eqfp} holds, then we must have $\pi
(\Gamma) \varphi= \{\varphi\}$.
This follows by strict convexity since $P_{\dag} f$ is an average of
elements of~$\mathcal H$, all with norm $\| f\| _{\mathcal H}$.
Since $\pi$ does not have fixed points, we have reached a
contradiction, and~\eqref{eqnonzerolim} cannot hold, completing the proof.
\end{pf}

\begin{theorem}\label{thmharmonicNPT}
Let $\Gamma$ be a group with finite, symmetric generating set $S
\subseteq\Gamma$,
and let $P$ be the transition kernel of the simple random walk
on the Cayley graph $\mathsf{Cay}(G;S)$. If $\Gamma$ does not have
property \textup{(T)},
then there exists a Hilbert space $\mathcal{H}$ and a nonconstant
$\Gamma$-equivariant
$P$-harmonic mapping from $\Gamma$ into $\mathcal H$.
\end{theorem}

\begin{pf}
We write $\langle\cdot,\cdot\rangle$ and $\| \cdot\| $ for the inner
product and norm on $\mathcal H$.
Let $\{\psi_j\}_{j=0}^{\infty}$ be a sequence in $\mathcal H$ with
\[
\frac{\| (I-P_{\dag})\psi_j\| ^2}{\langle\psi_j, (I-P_{\dag})\psi_j\rangle} \to0.
\]
The existence of such a sequence is the content of Theorem~\ref{thmpropTphi}.

Define $\Psi_j \dvtx \Gamma\to\mathcal{H}$ by
\[
\Psi_j(g) = \frac{\pi(g) \psi_j}{2 \langle\psi_j, (I-P_{\dag
})\psi_j \rangle},
\]
where we recall the definition of $P_{\dag}$ from \eqref{eqPaction}.
Then, for every $j = 0,1,\ldots,$ and for any $g \in\Gamma$,
\begin{eqnarray*}
\frac{1}{|S|} \sum_{s \in S} \bigl\| \Psi_j(g)-
\Psi_j(gs)\bigr\| ^2 &=& \frac
{1}{|S|} \sum
_{s \in S} \frac{\| \pi(g) \psi_j - \pi(gs) \psi_j\| ^2}{2 \langle\psi_j, (I-P_{\dag})\psi_j \rangle}\\
 &=& \frac{1}{|S|} \sum
_{s \in S} \frac{\| \psi_j - \pi(s) \psi_j\| ^2}{2 \langle\psi_j, (I-P_{\dag})\psi_j \rangle} = 1,
\end{eqnarray*}
where we have used the fact that $\pi$ acts by isometries.

By the same token,
\[
\biggl\llVert \Psi_j(g)-\frac{1}{|S|} \sum
_{s \in S} \Psi_j(gs) \biggr\rrVert^2 =
\frac{\| (I-P_{\dag}) \psi_j\| ^2}{2 \langle\psi_j, (I-P_{\dag
}) \psi_j\rangle} \to0.
\]
Equipping $\Gamma$ with the word metric on $\mathsf{Cay}(G;S)$,
an application of Lemma~\ref{lemgencon} finishes the proof, just as
in Theorem~\ref{thmharmonic}.
\end{pf}

\section{The rate of escape}
\label{secestimates}

We now show how simple estimates derived from harmonic functions lead
to more detailed
information about the random walk. In fact, we will show that in
general situations, a hitting time
bound alone leads to some finer estimates.

\subsection{Graph estimates}

Consider again a symmetric,
stochastic matrix $\{P(x,y)\}_{x,y \in V}$ for some
at most countable
index set $V$.
Let $\Gamma\leq\Aut(P)$ be a subgroup
that acts transitively on $V$, and let $\dist$ be
a $\Gamma$-invariant metric on $V$.
Finally, let $\{X_t\}$ denote the random walk
with transition kernel $P$ started at some fixed point $x_0 \in V$.


For any $k \in\mathbb N$, let $H_k$ denote the first
time $t$ at which $\dist(X_0, X_{t}) \geq k$, and define
the function $h \dvtx \mathbb N \to\mathbb R$ by $h(k)=\mathbb E[H_k]$.
We start with the following simple lemma which employs
reversibility, transitivity and the triangle inequality.
It is based on an observation due to Mark Braverman;
see also the closely related inequalities of Babai~\cite{Babai91}.

\begin{lemma}\label{lemmark}
For any $T \geq0$, we have
\[
\E \dist(X_0, X_{T}) \geq\frac12 \max_{0 \leq t \leq T} \E
\bigl[\dist(X_0, X_t)-\dist(X_0,
X_1) \bigr] .
\]
\end{lemma}

\begin{pf}
Let $s' \leq T$ be such that
\[
\E \dist(X_0,X_{s'}) = \max_{0 \leq t \leq T} \E
\dist(X_0, X_t) .
\]
Then there exists an even time $s \in\{s',s'-1\}$ such that $\E
\dist(X_0,X_s) \geq\E [\dist(X_0,X_{s'})-d(X_0,X_1) ]$.

Consider $\{X_t\}$ and an identically distributed walk $\{\tilde{X}_t\}$
such that $\tilde{X}_t=X_t$ for $t \leq s/2$, and $\tilde{X}_t$ evolves
independently after time $s/2$.
By the triangle inequality, we have
\[
\dist(X_0, \tilde{X}_{T}) + \dist(\tilde{X}_{T},
X_{s}) \geq\dist (X_0, X_{s}) .
\]
But by reversibility and transitivity, each of $\dist(X_0, \tilde
{X}_T)$ and $\dist(\tilde{X}_T, X_s)$ are
distributed as $\dist(X_0, X_T)$. Taking expectations, the claimed
result follows.
\end{pf}

\begin{lemma}
\label{lemlindrift}
If $h(k) \leq T$, then
\[
\E \dist(X_0, X_{2T}) \geq\frac{k}{4} - \frac14
\E \dist (X_0,X_1) .
\]
\end{lemma}

\begin{pf}
Let $\alpha= \frac{1}{k} \max_{0 \leq t \leq2T} \E \dist(X_0, X_t)$.
First, observe that the triangle inequality implies
%
\begin{equation}
\label{eqfirst} \dist(X_0, X_{2T}) \geq
\1_{\{H_k \leq2T\}} \bigl(k - \dist (X_{H_k}, X_{2T}) \bigr) .
\end{equation}
By transitivity, we also have
\begin{eqnarray*}
\E \bigl[\1_{\{H_k \leq2T\}} \cdot\dist(X_{H_k}, X_{2T})
\bigr] &=& \E \bigl[\1_{\{H_k \leq2T\}} \cdot\dist(X_0,
X_{2T-H_k}) \bigr]\\
&\leq&\pr(H_k \leq2T) \alpha k .
\end{eqnarray*}
Thus taking expectations in \eqref{eqfirst} yields
\[
\E \dist(X_0, X_{2T}) \geq\pr(H_k \leq2T)
(1-\alpha)k \geq\tfrac 12 (1-\alpha)k ,\vadjust{\goodbreak}
\]
recalling our assumption that $\mathbb E[H_k] \leq T$.
On the other hand, Lemma~\ref{lemmark} shows that
\[
\E \dist(X_0, X_{2T}) \geq\tfrac{1}{2} \bigl(
\alpha k - \E \bigl[\dist(X_0,X_1) \bigr] \bigr).
\]
Averaging these two inequalities yields the desired result.
\end{pf}

Using transitivity, one can also prove a small-ball occupation
estimate, directly from information
on the hitting times.

\begin{theorem}\label{thmsmallball}
Assume that $\dist(X_0, X_1) \leq B$ almost surely.
Consider any $k \in\mathbb N$.
If $h(k) \leq T$, then for any $\varepsilon> 0$, we have
\[
\frac{1}{T} \sum_{t=0}^{T} \pr
\bigl[\dist(X_0, X_t) \leq \varepsilon k \bigr] \leq O(
\varepsilon+B/k) .
\]
\end{theorem}

\begin{pf}
Let $\alpha= 3 \varepsilon k$.
We define a sequence of random times $\{t_i\}_{i=0}^{\infty}$ as follows.
First, $t_0=0$. We then define $t_{i+1}$ as the smallest time $t > t_i$
such that
$\dist(X_t, X_{t_j}) \geq\alpha$ for all $j \leq i$. We put
$t_{i+1}=\infty$ if no
such $t$ exists.
Observe that the set $\{ X_{t_i} \dvtx t_i < \infty\}$ is $\alpha
$-separated in the metric $\dist$.

We then define, for each $i \geq0$,
the quantity
\[
\tau_i = \cases{ 0, & \quad$\mbox{if } t_i > 2T$,
\vspace*{2pt}
\cr
\# \bigl\{ t \in[t_i, t_i + 2T] \dvtx
\dist(X_t, X_{t_i}) \leq\alpha/3 \bigr\}, &\quad $
\mbox{otherwise.}$ }
\]

Since the set $\{ X_{t_i} \dvtx t_i < \infty\}$ is $\alpha$-separated,
the $(\alpha/3)$-balls about
the centers $X_{t_i}$ are disjoint, and thus we have the inequality
%
\begin{equation}
\label{eqineq1} 4T \geq\sum_{i=0}^{\infty}
\tau_i ,
\end{equation}
where the latter sum is over only finitely many terms.

We can also calculate for any $i \geq0$,
\[
\E[\tau_i] \geq\pr(t_i \leq2T) \cdot\E[
\tau_0],
\]
using transitivity.
Now, we have $t_i \leq2T$ if $\dist(X_0, X_T) \geq(B+\alpha) i$,
and thus
for $i \leq k/(B+\alpha)$, we have
\[
\pr(t_i > 2T) \leq\pr(H_k > 2T) \leq\frac{\E[H_k]}{2T}
\leq\frac 12 .
\]
We conclude that for $i \leq k/(B+\alpha)$, we have $\E[\tau_i] \geq
\frac12 \E[\tau_0]$. Combining this with~\eqref{eqineq1} yields
\[
\E[\tau_0] \leq O \bigl((\varepsilon+B/k) T \bigr) .
\]
\upqed\end{pf}

\begin{remark}
In the next section, we prove analogs of the preceding statements for
Hilbert space-valued
martingales. One can then use harmonic functions to obtain such results
in the graph setting.\vadjust{\goodbreak}
The results in this section are somewhat more general though, since
they give
general connections between the function $h(k)$ and other properties of
the chain.
For instance, for every $j \in\mathbb N$, there are groups where $h(k)
\asymp k^{1/(1-2^{-j})}$
as $k \to\infty$~\cite{Ershler01}.
\end{remark}

\subsection{Martingale estimates}

We now prove analogs of Lemma~\ref{lemlindrift} and Theorem~\ref
{thmsmallball}
in the setting of martingales.

Let $\{M_t\}$ be a martingale taking values in some Hilbert space
$\mathcal H$,
with respect to the filtration $\{\mathcal F_t\}$. Assume
that
$\E [\| M_{t+1}-M_t\| ^2\mid\mathcal F_t ] \geq1$ for every
$t \geq0$, and
there exists a $B \geq1$ such that for every $t \geq0$, we have
$|M_{t+1}-M_t| \leq B$ almost surely.

\begin{lemma}[(Martingale hitting times)]
\label{lemmghit}
For $R \geq0$, let
$\tau$ be the first time $t$ such that $\| M_t-M_0\|  \geq R$.
Then, $\E(\tau) \leq(R+B)^2$.
\end{lemma}

\begin{pf}
Applying the optional stopping theorem (see, e.g.,~\cite{LPW09}, Corollary~17.7) to the submartingale $\| M_t-M_0\| ^2-t$, with the stopping
time $\tau$,
we see that
$\E(\tau) \leq\E(\| M_{\tau}-M_0\| ^2)$, and $\E(\| M_{\tau}-M_0\| ^2) \leq(R+B)^2.$
\end{pf}

The following simple estimate gives a lower bound on the $L^1$ rate of
escape for a martingale.

\begin{lemma}[($L^1$ estimate)]
\label{lemL1mg}
For every $T \geq0$, we have $\E \llVert M_0-M_T\rrVert  \geq\sqrt {T/8} - B/2$.
\end{lemma}

\begin{pf}
Let $\tau\geq0$ be the first time such that $\| M_0-M_{\tau}\|  \geq
\sqrt{T/2}-B$,
and let $\tau' = \min(\tau,T)$. First, by Lemma~\ref{lemmghit} and
Markov's inequality, we have
\[
\E \| M_0-M_{\tau'}\|  \geq\pr(\tau\leq T) \cdot (\sqrt
{T/2}-B ) \geq\sqrt{\frac{T}{8}} - \frac{B}{2} .
\]
Then, since $\| M_0-M_t\| $ is a submartingale and $\tau'$ and $T$ are
stopping times with $\tau' \leq T$, we have
\[
\E \| M_0-M_T\|  \geq\E \| M_0-M_{\tau'}
\|  .
\]
\upqed\end{pf}

Now we prove an analog of Theorem~\ref{thmsmallball} in the
martingale setting,
beginning with a preliminary lemma.

\begin{lemma}\label{lemyuval}
For $R \geq R' \geq0$, let $p_R$ denote the probability
that $\| M_t\|  \geq\| M_0\|  + R$ occurs before $\| M_t\|  \leq\| M_0\| -R'$.
Then $p_R \geq\frac{R'}{2R+B}$.
\end{lemma}

\begin{pf}
Let $\tau\geq0$ be the first time at which $\| M_{\tau}\|  \geq\|
M_0\|  + R$ or $\| M_{\tau}\|  \leq\| M_0\| -R'$.\vadjust{\goodbreak}
Since $\| M_t\| -\| M_0\| $ is a submartingale, the optional stopping
theorem implies
\[
0 \leq\E \bigl(\| M_{\tau}\| -\| M_0\|  \bigr) \leq p_R
(R+B) - (1-p_R) R' \leq p_R(2R+B) -
R' .
\]
Rearranging yields the desired result.
\end{pf}

From this, we can prove a general occupation time estimate.

\begin{lemma}[(Martingale occupation times)]
\label{lemmgocc}
Suppose that $M_0=0$. Then for every $\eps\geq B/\sqrt{T}$ and $T
\geq1$, we have
\[
\frac{1}{T} \sum_{t=0}^T \pr \bigl[
\| M_t\|  \leq\eps\sqrt{T} \bigr] \leq O(\eps) .
\]
\end{lemma}

\begin{pf}
Let $h=\lceil2 (3 \varepsilon\sqrt{T} + B)^2\rceil$.
Let $\mathcal B$ denote the ball of radius $\varepsilon\sqrt{T}$
about $0$ in $\mathcal H$.
For $t \leq T-h$,
let $p_t$ denote the probability that $M_t \in\mathcal B$,
but $M_{t+h}, M_{t+h+1}, \ldots, M_{T} \notin\mathcal B$.
We first show that for every such $t$,
%
\begin{equation}
\label{eqwantshow} p_t \geq\frac{\varepsilon}{40} \cdot\pr \bigl(
\| M_t\|  \leq \varepsilon\sqrt{T} \bigr) .
\end{equation}

To this end, we prove three bounds. First,
%
\begin{eqnarray}
\label{eqfone} &&\pr \bigl(\exists i \leq h  \mbox{ such that } \| M_{t+i}\|
\geq2 \varepsilon\sqrt{T} | M_t \in\mathcal B \bigr)
\nonumber
\\[-8pt]
\\[-8pt]
\nonumber
& &\qquad\geq\pr \bigl(\exists i \leq h \mbox{ such that } \|  M_{t+i}-M_t
\|  \geq3 \varepsilon\sqrt{T} | M_t \in\mathcal B \bigr) \geq\tfrac12 ,
\end{eqnarray}
where the latter bound follows from Markov's inequality and Lemma \ref
{lemmghit}.

Next, observe that for $R \geq\varepsilon\sqrt{T}$, Lemma \ref
{lemyuval} gives
%
\begin{eqnarray}
\label{eqsecondstep}\qquad&& \pr \bigl(\| M_s\|  \geq R \mbox{ occurs before }
\| M_s\|  \leq \varepsilon\sqrt{T} \mbox{ for } s \geq t+i |
\| M_{t+i}\|  \geq2 \varepsilon\sqrt{T} \bigr)
\nonumber
\\[-8pt]
\\[-8pt]
\nonumber
&&\qquad\geq\frac{\varepsilon\sqrt {T}}{2R+B} .
\end{eqnarray}
Finally, the Doob--Kolmogorov maximal inequality implies that
%
\begin{equation}
\label{eqdoob} \pr \Bigl(\max_{0 \leq r \leq T} \| M_s-M_{s+r}
\|  > R | \mathcal F_s \Bigr) \leq\frac{\E [\| M_s-M_{s+T}\| ^2\mid
\mathcal F_s ]}{R^2} =
\frac{T}{R^2} .
\end{equation}
Setting $R=2\sqrt{T}$, \eqref{eqsecondstep} and \eqref{eqdoob}
imply that for any time $t \geq0$, we have
\[
\pr \bigl(M_{t+i}, M_{t+i+1}, \ldots, M_T \notin
\mathcal B | \| M_{t+i}\|  \geq2 \varepsilon\sqrt{T} \bigr) \geq
\frac
{\varepsilon}{20} .
\]
Combining this with \eqref{eqfone} yields \eqref{eqwantshow}.

But we must have
\[
\sum_{t=0}^T p_t \leq h = O
\bigl(\varepsilon^2 T \bigr),
\]
by construction.
Thus \eqref{eqwantshow} yields
\[
\sum_{t=0}^T \pr \bigl[\| M_t\|
\leq\varepsilon\sqrt{T} \bigr] \leq O(\varepsilon T),
\]
completing the proof.
\end{pf}

\subsection{Applications}
\label{secapplications}

Combining the observations of the preceding section,
together with the existence of harmonic functions, yields
our claimed results on transitive graphs.
In particular, the following result, combined
with Theorem~\ref{thmharmonic}, proves Theorem~\ref{thmmaininfinite}.

\begin{theorem}
\label{thminfdiffuse}
Let $V$ be a countably infinite index set, and
let\break $\{P(x,y)\}_{x,y \in V}$ be a
stochastic, symmetric matrix.
Suppose that $\Gamma\leq\Aut(P)$
is a closed, amenable, unimodular subgroup that
acts transitively on $V$,
and
there exists a connected graph $G=(V,E)$ on which
$\Gamma$ acts by automorphisms.
Suppose further that for some $B > 0$, for all $x,y \in V$, we have
%
\begin{equation}
\label{eqbddstep2} P(x,y) \qquad\mbox{implies } \dist(x,y) \leq B ,
\end{equation}
where $\dist$ is the path metric on $G$.
Suppose also that
\[
p_* = \min \bigl\{ P(x,y) \dvtx \{x,y\} \in E \bigr\} > 0.
\]
If there exists a Hilbert space $\mathcal H$
and a nonconstant $\Gamma$-equivariant $\mathcal H$-valued
$P$-harmonic mapping, then the following holds.

Let $\{X_t\}$ denote the random walk with transition kernel $P$. For
every $t \geq0$, we
have the estimates
%
\begin{eqnarray}
\label{eqesc1} \E \bigl[\dist(X_0, X_t)^2
\bigr] &\geq& p_* t, \label{eqesc2}
\\
\E \bigl[\dist(X_0, X_t) \bigr] &\geq&
\frac{\sqrt{p_* t}}{24} - \frac32 B ,
\end{eqnarray}
and, for every $\varepsilon\geq1/\sqrt{T}$ and $T \geq4/p_*$,
%
\begin{equation}
\label{eqocc3} \frac{1}{T} \sum_{t=0}^{T}
\pr \bigl[\dist(X_0, X_t) \leq \varepsilon\sqrt{p_* T/B}
\bigr] \leq O(\varepsilon) .
\end{equation}
\end{theorem}

\begin{pf}
Let $\mathcal H$ and $\Psi: V \to\mathcal H$ be the Hilbert
space and nonconstant $\Gamma$-equivariant $P$-harmonic mapping.
Let $\| \cdot\|  = \| \cdot\| _{\mathcal H}$, and
normalize $\Psi$ so that
for every $x \in V$,
%
\begin{equation}
\label{eqpsinorm} \sum_{y \in V} P(x,y) \bigl\| \Psi(x)-
\Psi(y)\bigr\| ^2 = 1.
\end{equation}
Then $M_t = \Psi(X_t)$ is an $\mathcal H$-valued martingale
with $\E [\| M_{t+1}-M_t\| ^2 | \mathcal F_t ] = 1$ for every
$t \geq0$.

Furthermore, from \eqref{eqpsinorm},
we see that $\Psi$ is $\sqrt{1/p_*}$-Lipschitz
as a mapping from $(V,\dist)$ to $\mathcal H$.
Thus one has immediately the estimate
\[
\E \bigl[\dist(X_0,X_t)^2 \bigr] \geq p_*
\E \bigl[\| M_t-M_0\| ^2 \bigr] = p_* t .
\]

Now, for any $k \in\mathbb N$, let $H_k$ denote the first
time $t$ at which $\dist(X_0, X_{t})=k$, and define
the function $h \dvtx \mathbb N \to\mathbb R$ by $h(k)=\mathbb E[H_k]$.
Since $\Psi$ is $\sqrt{1/p_*}$-Lipschitz, Lemma~\ref{lemmghit}
applied to $\{M_t\}$ shows that for every $k \in\mathbb N$,
\[
h(k) \leq\frac{(k+B)^2}{p_*} .
\]
Combining this with Lemma~\ref{lemlindrift} yields~\eqref{eqesc1}.
Combining it with Theorem~\ref{thmsmallball} yields~\eqref{eqocc3}.
\end{pf}

Although we have proved a result about occupation times, we conjecture
that a stronger bound holds.

\begin{conjecture}
Suppose that $G$ is an infinite, transitive, connected, amenable graph
with degree $d$, and $\{X_t\}$ is the simple random walk on $G$.
Theorem~\ref{thminfdiffuse} shows that
for every $\varepsilon> 1/\sqrt{T}$ and $T \geq4 d$, we have
\[
\frac{1}{T} \sum_{t=0}^{T} \pr
\bigl(\dist(X_0, X_t) \leq\varepsilon \sqrt{T/d} \bigr)
\leq O(\eps) .
\]
We conjecture that this holds pointwise; that is, for every large
enough time~$t$, we have
\[
\pr \bigl(\dist(X_0, X_t) \leq\varepsilon\sqrt{t/d}
\bigr) \leq O(\eps) .
\]
\end{conjecture}

Finally, we conclude with a theorem about finite graphs which,
in particular, yields Theorem~\ref{thmmainfinite}.

\begin{theorem}
Let $V$ be a finite index set
and
suppose that $\Aut(P)$ acts transitively on $V$,
and on the graph $G=(V,E)$ by automorphisms.
If
\[
p_* = \min \bigl\{ P(x,y) \dvtx \{x,y\} \in E \bigr\} > 0,
\]
and $\lambda< 1$ is the second-largest eigenvalue of $P$, then
for every $t \leq(1-\lambda)^{-1}$, we have
\begin{eqnarray*}
\E \bigl[\dist(X_0, X_t)^2 \bigr] &\geq&
p_* t/2 ,
\\
\E \bigl[\dist(X_0, X_t) \bigr] &\geq& \Omega({
\sqrt{p_* t}}) - B ,
\end{eqnarray*}
and, for every $\varepsilon> 0$ and $(1-\lambda)^{-1} \geq T \geq4/p_*$,
%
\begin{equation}
\label{eqocc32} \frac{1}{T} \sum_{t=0}^{T}
\pr \bigl[\dist(X_0, X_t) \leq \varepsilon\sqrt{p_* T/B}
\bigr] \leq O(\varepsilon) .
\end{equation}
\end{theorem}

\begin{pf}
Let $\psi: V \to\mathbb R$ be such that $P \psi= \lambda\psi$, and
define $\Psi: V \to\ell^2(\Aut(P))$
by
\[
\Psi(x) = \frac{ (\psi(\sigma x) )_{\sigma\in\Aut(P)}}{\sqrt{2
\langle\psi, (I-P)\psi\rangle}}.
\]
An argument as in \eqref{eqgrad} shows that $\| \Psi\| _{\Lip} \leq
\sqrt{1/p_*}$.

Now, observe that $\{\lambda^{-t} \Psi(X_t)\}$ is a martingale.
This follows from the fact that $\lambda^{-t} \psi(X_t)$ is a martingale,
which one easily checks.
\[
\E \bigl[\lambda^{-t-1} \psi(X_{t+1}) | X_t \bigr]
= \lambda^{-t-1} (P \psi) (X_t) = \lambda^{-t}
\psi(X_t).
\]
Note that $t \leq(1-\lambda)^{-1}$, hence the mapping
$x \mapsto\lambda^{-t} \Psi(x)$ is $O(\sqrt{1/p_*})$-Lipschitz,
and the same argument as in Theorem~\ref{thminfdiffuse} applies.
\end{pf}

\section*{Acknowledgments}
We thank Tonci Antunovic for detailed comments
on earlier drafts of this manuscript, and
Tim Austin for his suggestions toward obtaining equivariance
in Theorem~\ref{thmharmonic}. We also thank Anna \`{E}rshler, David
Fisher, Subhroshekhar Gosh, Gady Kozma, G\'abor Pete and
B\'alint Vir\'ag for useful discussions. Finally, we are grateful to an
anonymous referee
for many detailed suggestions.


%


\printaddresses


\begin{thebibliography}{30}

\bibitem{ANP}
\begin{barticle}[mr]
\bauthor{\bsnm{Austin},~\bfnm{Tim}\binits{T.}},
  \bauthor{\bsnm{Naor},~\bfnm{Assaf}\binits{A.}} \AND
  \bauthor{\bsnm{Peres},~\bfnm{Yuval}\binits{Y.}}
(\byear{2009}).
\btitle{The wreath product of {$\Bbb Z$} with {$\Bbb Z$} has {H}ilbert
  compression exponent {$\frac{2}{3}$}}.
\bjournal{Proc. Amer. Math. Soc.}
\bvolume{137}
\bpages{85--90}.%
\bid{doi={10.1090/S0002-9939-08-09501-4}, issn={0002-9939}, mr={2439428}}%
\bptok{imsref}%
\end{barticle}%
\endbibitem

\bibitem{Babai91}
\begin{bincollection}[auto:STB|2012/09/25|13:49:33]
\bauthor{\bsnm{Babai},~\bfnm{L.}\binits{L.}}
(\byear{1991}).
\btitle{Local expansion of vertex-transitive graphs and random generation in
  finite groups}.
In \bbooktitle{Proceedings of the Twenty Third Annual ACM Symposium on Theory
  of Computing}
\bpages{164--174}.
\bpublisher{ACM}, \blocation{New York}.
\bptok{imsref}%
\end{bincollection}
\endbibitem

\bibitem{BHV08}
\begin{bbook}[auto:STB|2012/09/25|13:49:33]
\bauthor{\bsnm{Bekka},~\bfnm{B.}\binits{B.}}, \bauthor{\bparticle{de~la}
  \bsnm{Harpe},~\bfnm{P.}\binits{P.}} \AND
  \bauthor{\bsnm{Valette},~\bfnm{A.}\binits{A.}}
(\byear{2008}).
\btitle{{Kazhdan's Property (T)}, Volume 11 of \textit{New Mathematical
  Monographs}}.
\bpublisher{Cambridge Univ. Press}, \blocation{Cambridge}.
\bptok{imsref}%
\end{bbook}
\endbibitem

\bibitem{BLPS99}
\begin{barticle}[mr]
\bauthor{\bsnm{Benjamini},~\bfnm{I.}\binits{I.}},
  \bauthor{\bsnm{Lyons},~\bfnm{R.}\binits{R.}},
  \bauthor{\bsnm{Peres},~\bfnm{Y.}\binits{Y.}} \AND
  \bauthor{\bsnm{Schramm},~\bfnm{O.}\binits{O.}}
(\byear{1999}).
\btitle{Group-invariant percolation on graphs}.
\bjournal{Geom. Funct. Anal.}
\bvolume{9}
\bpages{29--66}.
\bid{doi={10.1007/s000390050080}, issn={1016-443X}, mr={1675890}}
\bptok{imsref}%
\end{barticle}
\endbibitem

\bibitem{DSC94}
\begin{barticle}[mr]
\bauthor{\bsnm{Diaconis},~\bfnm{P.}\binits{P.}} \AND
  \bauthor{\bsnm{Saloff-Coste},~\bfnm{L.}\binits{L.}}
(\byear{1994}).
\btitle{Moderate growth and random walk on finite groups}.
\bjournal{Geom. Funct. Anal.}
\bvolume{4}
\bpages{1--36}.
\bid{doi={10.1007/BF01898359}, issn={1016-443X}, mr={1254308}}
\bptok{imsref}%
\end{barticle}
\endbibitem

\bibitem{Ershler01}
\begin{barticle}[mr]
\bauthor{\bsnm{{\`E}rshler},~\bfnm{A.~G.}\binits{A.~G.}}
(\byear{2001}).
\btitle{On the asymptotics of the rate of departure to infinity}.
\bjournal{Zap. Nauchn. Sem. S.-Peterburg. Otdel. Mat. Inst. Steklov. (POMI)}
\bvolume{283}
\bpages{251--257, 263}.
\bid{doi={10.1023/B:JOTH.0000024624.22696.52}, issn={0373-2703}, mr={1879073}}
\bptok{imsref}%
\end{barticle}
\endbibitem

\bibitem{Erschler08}
\begin{bmisc}[auto:STB|2012/09/25|13:49:33]
\bauthor{\bsnm{{\`E}rshler},~\bfnm{A.~G.}\binits{A.~G.}}
(\byear{2005}).
\bhowpublished{Personal communication}.
\bptok{imsref}%
\end{bmisc}
\endbibitem

\bibitem{FM05}
\begin{barticle}[mr]
\bauthor{\bsnm{Fisher},~\bfnm{David}\binits{D.}} \AND
  \bauthor{\bsnm{Margulis},~\bfnm{Gregory}\binits{G.}}
(\byear{2005}).
\btitle{Almost isometric actions, property ({T}), and local rigidity}.
\bjournal{Invent. Math.}
\bvolume{162}
\bpages{19--80}.
\bid{doi={10.1007/s00222-004-0437-5}, issn={0020-9910}, mr={2198325}}
\bptok{imsref}%
\end{barticle}
\endbibitem

\bibitem{GS92}
\begin{bbook}[mr]
\bauthor{\bsnm{Grimmett},~\bfnm{G.~R.}\binits{G.~R.}} \AND
  \bauthor{\bsnm{Stirzaker},~\bfnm{D.~R.}\binits{D.~R.}}
(\byear{1992}).
\btitle{Probability and Random Processes},
\bedition{2nd} ed.
\bpublisher{Clarendon Press Oxford Univ. Press}, \blocation{New York}.
\bid{mr={1199812}}
\bptok{imsref}%
\end{bbook}
\endbibitem

\bibitem{Gromov03}
\begin{barticle}[mr]
\bauthor{\bsnm{Gromov},~\bfnm{M.}\binits{M.}}
(\byear{2003}).
\btitle{Random walk in random groups}.
\bjournal{Geom. Funct. Anal.}
\bvolume{13}
\bpages{73--146}.
\bid{doi={10.1007/s000390300002}, issn={1016-443X}, mr={1978492}}
\bptok{imsref}%
\end{barticle}
\endbibitem

\bibitem{HSC}
\begin{barticle}[mr]
\bauthor{\bsnm{Hebisch},~\bfnm{W.}\binits{W.}} \AND
  \bauthor{\bsnm{Saloff-Coste},~\bfnm{L.}\binits{L.}}
(\byear{1993}).
\btitle{Gaussian estimates for {M}arkov chains and random walks on groups}.
\bjournal{Ann. Probab.}
\bvolume{21}
\bpages{673--709}.
\bid{issn={0091-1798}, mr={1217561}}
\bptok{imsref}%
\end{barticle}
\endbibitem

\bibitem{Kesten59}
\begin{barticle}[mr]
\bauthor{\bsnm{Kesten},~\bfnm{Harry}\binits{H.}}
(\byear{1959}).
\btitle{Symmetric random walks on groups}.
\bjournal{Trans. Amer. Math. Soc.}
\bvolume{92}
\bpages{336--354}.
\bid{issn={0002-9947}, mr={0109367}}
\bptok{imsref}%
\end{barticle}
\endbibitem

\bibitem{Kleiner07}
\begin{barticle}[mr]
\bauthor{\bsnm{Kleiner},~\bfnm{Bruce}\binits{B.}}
(\byear{2010}).
\btitle{A new proof of {G}romov's theorem on groups of polynomial growth}.
\bjournal{J.~Amer. Math. Soc.}
\bvolume{23}
\bpages{815--829}.
\bid{doi={10.1090/S0894-0347-09-00658-4}, issn={0894-0347}, mr={2629989}}
\bptok{imsref}%
\end{barticle}
\endbibitem

\bibitem{KS97}
\begin{barticle}[mr]
\bauthor{\bsnm{Korevaar},~\bfnm{Nicholas~J.}\binits{N.~J.}} \AND
  \bauthor{\bsnm{Schoen},~\bfnm{Richard~M.}\binits{R.~M.}}
(\byear{1997}).
\btitle{Global existence theorems for harmonic maps to non-locally compact
  spaces}.
\bjournal{Comm. Anal. Geom.}
\bvolume{5}
\bpages{333--387}.
\bid{issn={1019-8385}, mr={1483983}}
\bptok{imsref}%
\end{barticle}
\endbibitem

\bibitem{KW92}
\begin{bbook}[mr]
\bauthor{\bsnm{Kwapie{\'n}},~\bfnm{Stanis{\l}aw}\binits{S.}} \AND
  \bauthor{\bsnm{Woyczy{\'n}ski},~\bfnm{Wojbor~A.}\binits{W.~A.}}
(\byear{1992}).
\btitle{Random Series and Stochastic Integrals: Single and Multiple}.
\bpublisher{Birkh\"auser}, \blocation{Boston, MA}.
\bid{mr={1167198}}
\bptok{imsref}%
\end{bbook}
\endbibitem

\bibitem{LM}
\begin{bmisc}[auto]
\bauthor{\bsnm{Lee},~\bfnm{James~R.}\binits{J.~R.}} \AND
  \bauthor{\bsnm{Makarychev},~\bfnm{Y.}\binits{Y.}}
(\byear{2009}).
\bhowpublished{Eigenvalue multiplicity and volume growth. Available at
  arXiv:\arxivurl{0806.1745} [math.MG].}
\bptok{imsref}%
\end{bmisc}
\endbibitem

\bibitem{LPW09}
\begin{bbook}[mr]
\bauthor{\bsnm{Levin},~\bfnm{David~A.}\binits{D.~A.}},
  \bauthor{\bsnm{Peres},~\bfnm{Yuval}\binits{Y.}} \AND
  \bauthor{\bsnm{Wilmer},~\bfnm{Elizabeth~L.}\binits{E.~L.}}
(\byear{2009}).
\btitle{Markov Chains and Mixing Times}.
\bpublisher{Amer. Math. Soc.}, \blocation{Providence, RI}.
\bid{mr={2466937}}
\bptok{imsref}%
\end{bbook}
\endbibitem

\bibitem{LS99}
\begin{barticle}[mr]
\bauthor{\bsnm{Lyons},~\bfnm{Russell}\binits{R.}} \AND
  \bauthor{\bsnm{Schramm},~\bfnm{Oded}\binits{O.}}
(\byear{1999}).
\btitle{Stationary measures for random walks in a random environment with
  random scenery}.
\bjournal{New York J. Math.}
\bvolume{5}
\bpages{107--113 (electronic)}.
\bid{issn={1076-9803}, mr={1703207}}
\bptok{imsref}%
\end{barticle}
\endbibitem

\bibitem{Mok95}
\begin{binproceedings}[mr]
\bauthor{\bsnm{Mok},~\bfnm{Ngaiming}\binits{N.}}
(\byear{1995}).
\btitle{Harmonic forms with values in locally constant {H}ilbert bundles}.
In \bbooktitle{Proceedings of the {C}onference in {H}onor of {J}ean-{P}ierre
  {K}ahane ({O}rsay, 1993). J. Fourier Anal. Appl. Special Issue}
\bpages{433--453}.
\bid{issn={1069-5869}, mr={1364901}}
\bptok{imsref}%
\end{binproceedings}
\endbibitem

\bibitem{NP1}
\begin{barticle}[mr]
\bauthor{\bsnm{Naor},~\bfnm{Assaf}\binits{A.}} \AND
  \bauthor{\bsnm{Peres},~\bfnm{Yuval}\binits{Y.}}
(\byear{2008}).
\btitle{Embeddings of discrete groups and the speed of random walks}.
\bjournal{Int. Math. Res. Not. IMRN}
\bpages{Art. ID rnn 076, 34}.
\bid{issn={1073-7928}, mr={2439557}}
\bptok{imsref}%
\end{barticle}
\endbibitem

\bibitem{NP2}
\begin{barticle}[mr]
\bauthor{\bsnm{Naor},~\bfnm{Assaf}\binits{A.}} \AND
  \bauthor{\bsnm{Peres},~\bfnm{Yuval}\binits{Y.}}
(\byear{2011}).
\btitle{{$L\sb p$} compression, traveling salesmen, and stable walks}.
\bjournal{Duke Math. J.}
\bvolume{157}
\bpages{53--108}.
\bid{doi={10.1215/00127094-2011-002}, issn={0012-7094}, mr={2783928}}
\bptok{imsref}%
\end{barticle}
\endbibitem

\bibitem{PPS06}
\begin{barticle}[mr]
\bauthor{\bsnm{Peres},~\bfnm{Yuval}\binits{Y.}},
  \bauthor{\bsnm{Pete},~\bfnm{G{\'a}bor}\binits{G.}} \AND
  \bauthor{\bsnm{Scolnicov},~\bfnm{Ariel}\binits{A.}}
(\byear{2006}).
\btitle{Critical percolation on certain nonunimodular graphs}.
\bjournal{New York J. Math.}
\bvolume{12}
\bpages{1--18 (electronic)}.
\bid{issn={1076-9803}, mr={2217160}}
\bptok{imsref}%
\end{barticle}
\endbibitem

\bibitem{Salvatori92}
\begin{barticle}[mr]
\bauthor{\bsnm{Salvatori},~\bfnm{Maura}\binits{M.}}
(\byear{1992}).
\btitle{On the norms of group-invariant transition operators on graphs}.
\bjournal{J.~Theoret. Probab.}
\bvolume{5}
\bpages{563--576}.
\bid{doi={10.1007/BF01060436}, issn={0894-9840}, mr={1176438}}
\bptok{imsref}%
\end{barticle}
\endbibitem

\bibitem{ST}
\begin{barticle}[mr]
\bauthor{\bsnm{Shalom},~\bfnm{Yehuda}\binits{Y.}} \AND
  \bauthor{\bsnm{Tao},~\bfnm{Terence}\binits{T.}}
(\byear{2010}).
\btitle{A finitary version of {G}romov's polynomial growth theorem}.
\bjournal{Geom. Funct. Anal.}
\bvolume{20}
\bpages{1502--1547}.
\bid{doi={10.1007/s00039-010-0096-1}, issn={1016-443X}, mr={2739001}}
\bptok{imsref}%
\end{barticle}
\endbibitem

\bibitem{SW90}
\begin{barticle}[mr]
\bauthor{\bsnm{Soardi},~\bfnm{Paolo~M.}\binits{P.~M.}} \AND
  \bauthor{\bsnm{Woess},~\bfnm{Wolfgang}\binits{W.}}
(\byear{1990}).
\btitle{Amenability, unimodularity, and the spectral radius of random walks on
  infinite graphs}.
\bjournal{Math. Z.}
\bvolume{205}
\bpages{471--486}.
\bid{doi={10.1007/BF02571256}, issn={0025-5874}, mr={1082868}}
\bptok{imsref}%
\end{barticle}
\endbibitem

\bibitem{Var}
\begin{barticle}[mr]
\bauthor{\bsnm{Varopoulos},~\bfnm{N.~Th.}\binits{N.~T.}}
(\byear{1985}).
\btitle{Isoperimetric inequalities and {M}arkov chains}.
\bjournal{J. Funct. Anal.}
\bvolume{63}
\bpages{215--239}.
\bid{doi={10.1016/0022-1236(85)90086-2}, issn={0022-1236}, mr={0803093}}
\bptok{imsref}%
\end{barticle}
\endbibitem

\bibitem{Vershik00}
\begin{barticle}[mr]
\bauthor{\bsnm{Vershik},~\bfnm{A.~M.}\binits{A.~M.}}
(\byear{2000}).
\btitle{Dynamic theory of growth in groups: Entropy, boundaries, examples}.
\bjournal{Uspekhi Mat. Nauk}
\bvolume{55}
\bpages{59--128}.
\bid{doi={10.1070/rm2000v055n04ABEH000314}, issn={0042-1316}, mr={1786730}}
\bptok{imsref}%
\end{barticle}
\endbibitem

\bibitem{Virag08}
\begin{bmisc}[auto:STB|2012/09/25|13:49:33]
\bauthor{\bsnm{Virag},~\bfnm{B.}\binits{B.}}
(\byear{2005}).
\bhowpublished{Personal communication}.
\bptok{imsref}%
\end{bmisc}
\endbibitem

\bibitem{Woess00}
\begin{bbook}[mr]
\bauthor{\bsnm{Woess},~\bfnm{Wolfgang}\binits{W.}}
(\byear{2000}).
\btitle{Random Walks on Infinite Graphs and Groups}.
\bseries{Cambridge Tracts in Mathematics}
\bvolume{138}.
\bpublisher{Cambridge Univ. Press}, \blocation{Cambridge}.
\bid{doi={10.1017/CBO9780511470967}, mr={1743100}}
\bptok{imsref}%
\end{bbook}
\endbibitem

\end{thebibliography}
\end{document}